\documentclass[oneside,12pt]{amsart}

\usepackage{eurosym}
\usepackage{amsfonts}
\usepackage{amsmath,amssymb,amsthm,color}
\usepackage{amsopn}
\usepackage{latexsym}
\usepackage{float}
\usepackage{bm}
\usepackage[utf8]{inputenc}
\usepackage{hhline}
\usepackage[section]{placeins}

\newcommand{\mg}{\mathfrak g }

\newcommand{\mm}{\mathfrak m }
\newcommand{\mmu}{\mathfrak u }
\newcommand{\mz}{\mathfrak z }

\newcommand{\mk}{\mathfrak k }

\newcommand{\mv}{\mathfrak v }
\newcommand{\mh}{\mathfrak h }

\newcommand{\tmg}{\tilde{\mathfrak g}}
\newcommand{\so}{\mathfrak{so} }

\renewcommand{\sl}{\mathfrak{sl} }
\newcommand{\gl}{\mathfrak{gl} }

\newcommand{\lela}{ g(}
\newcommand{\rira}{)}

\newcommand{\bs}{\backslash}

\newcommand{\GL}{\operatorname{GL}}
\newcommand{\SL}{\operatorname{SL}}

\newcommand{\R}{\mathbb R}
  
\newcommand{\Q}{\mathbb Q}
\newcommand{\N}{\mathbb N}
\newcommand{\Z}{\mathbb Z}
\newcommand{\C}{\mathbb C}

\newcommand{\spa}{\mathrm{span}}

\DeclareMathOperator{\Aut}{Aut}

\DeclareMathOperator{\ad}{ad}

\DeclareMathOperator{\tr}{tr}

\numberwithin{equation}{section}

 \newtheorem{teo}{Theorem}[section]
 \newtheorem{pro}[teo]{Proposition}
 \newtheorem{cor}[teo]{Corollary}
 \newtheorem{lm}[teo]{Lemma}
 \newtheorem{defi}[teo]{Definition}

 \theoremstyle{definition}
 \newtheorem{ex}[teo]{Example}
 \newtheorem{remark}[teo]{Remark}

\newcommand{\nc}{\newcommand}
\nc{\Iso}{\operatorname{Iso}}
\nc{\Id}{\operatorname{Id}}
 \nc{\iso}{\mathfrak{iso}}
 \nc{\sso}{\mathfrak{so}}
\nc{\Ad}{\operatorname{Ad}} 
\nc{\Sym}{\mathrm{Sym}}
 \nc{\pr}{\operatorname{pr}} 
 \nc{\Dera}{\operatorname{Dera}} 
 \nc{\Auto}{\operatorname{Auto}}
 \nc{\noi}{\noindent}

\title{Locally conformally product structures on solvmanifolds}

\author{Adrián Andrada}
\address{A. Andrada: FAMAF, Universidad Nacional de C\'ordoba and CIEM-CONICET, Av. Medina Allende s/n, Ciudad Universitaria, X5000HUA C\'ordoba, Argentina. }
\email{adrian.andrada@unc.edu.ar}

\author{Viviana del Barco}
\address{V.~del Barco: Instituto de Matemática, Estatística e Computação Científica, Universidade Estadual de Campinas,  Rua Sergio Buarque de Holanda, 651, Cidade Universitaria Zeferino Vaz, 13083-859, Campinas, São Paulo, Brazil.}
\email{delbarc@ime.unicamp.br}

\author{Andrei Moroianu}
\address{A.~Moroianu: Université Paris-Saclay, CNRS,  Laboratoire de mathématiques d'Orsay, 91405, Orsay, France, and Institute of Mathematics “Simion Stoilow” of the Romanian Academy, 21
Calea Grivitei, 010702 Bucharest, Romania}
\email{andrei.moroianu@math.cnrs.fr}


\begin{document}

\begin{abstract} We study left invariant locally conformally product structures on simply connected Lie groups and give their complete description in the solvable unimodular case. Based on previous classification results, we then obtain the complete list of solvable unimodular Lie algebras up to dimension 5 which carry LCP structures, and study the existence of lattices in the corresponding simply connected Lie groups.
\end{abstract}

\maketitle

\section{Introduction}

A locally conformally product (LCP) structure on a compact connected manifold $M$ is a non-flat Riemannian metric $h$ with reducible holonomy on the universal cover $\tilde M$ such that $\pi_1(M)$ acts by homotheties with respect to $h$, not of all which are isometries. In particular $(\tilde M,h)$ has to be incomplete, so the fundamental group of $M$ is infinite.

It is worth noting that there is a strong similarity between LCP structures and locally conformally Kähler (LCK) structures, which is reflected in the terminology used.  Indeed, an LCK structure on a compact connected manifold $M$ can be defined as a Kähler metric $h$ on the universal cover $\tilde M$ with respect to which $\pi_1(M)$ acts by homotheties, not of all which are isometries. 

Just like in the LCK setting, an LCP manifold is not canonically endowed with a Riemannian metric. Instead, it carries a canonical conformal structure $c$ (induced by $h$) and a closed, non-exact Weyl structure $D$ with reducible holonomy, which is the projection of the Levi-Civita connection of $h$.

Conversely, every conformal manifold $(M,c)$ with a closed non-exact Weyl structure $D$ with reducible but non-flat holonomy is LCP, the metric $h$ on the universal cover being defined as the unique (up to a scalar factor) Riemannian metric on $\tilde M$ whose Levi-Civita connection is the lift of $D$. 

Note that LCP structures are not easy to construct. In fact it was conjectured in \cite{BM2016} that they simply do not exist, based on some evidence related to a result of Gallot on the irreducibility or flatness of cone metrics over compact manifolds \cite{Gallot}, and the fact that every closed non-exact {\em tame} Weyl connection is irreducible or flat. The tameness condition here is related to the life-time of incomplete geodesics of the connection (see \cite[Definition 3.2]{BM2016} for details).

However, shortly after, an example of LCP structure was constructed by Matveev and Nikolayevsky on a 3-dimensional solvmanifold \cite{MN2015}. The same authors then proved that in the analytic setting, the universal cover $(\tilde M,h)$ of any LCP manifold is globally isometric   
to a Riemannian product $\R^q\times (N,g_N)$ where $\R^q$ is a flat Euclidean space of dimension $q\ge 1$ and $(N,g_N)$ is a (necessarily incomplete) Riemannian manifold \cite{MN2017}. The analiticity assumption was later removed by Kourganoff \cite{Kourganoff}, who also showed that the metric of the non-flat factor $N$ is irreducible.

Note that although the universal cover $(\tilde M,h)$ of any LCP manifold is a simply connected Riemannian manifold with reducible holonomy, the global de Rham decomposition theorem does not hold in the incomplete setting, so it is already remarkable that $(\tilde M,h)$ has a global product decomposition. Even more striking, Kourganoff's result says that one of the factors of this product is flat and complete, and the other factor is irreducible.

It turns out that examples of LCP manifolds existed in the literature since 2005, but were not recognized as such until recently, since they were defined within the realm of complex geometry. These examples are the so-called OT manifolds, introduced by Oeljeklaus and Toma \cite{OT2005} by means of number fields with $s$ real embeddings and $2t$ complex embeddings. For $t=1$ the corresponding OT manifolds carry LCK structures, whose corresponding Kähler metrics are reducible (with a flat factor of dimension 2). In fact, building on previous works by Oeljeklaus and Toma \cite{OT2005} and Dubickas \cite{Du14}, Deaconu and Vuletescu \cite{DV22} have recently shown that an OT manifold carries LCK metrics if and only if $t=1$.

 LCP structures have been thoroughly investigated by Flamencourt, who showed that every OT manifold carries LCP structures \cite[Corollary 4.6]{Fl22} and constructed several new families of examples using number fields theory, generalising the OT examples. He also proved a number of interesting features of LCP structures, like the fact that all homothety factors of the action of the fundamental group on the universal cover are algebraic numbers \cite[Proposition 3.11]{Fl22}. This is a remarkable difference with the case of LCK manifolds, where no restriction exists on the homothety factors. 

It was noticed by Kasuya \cite[Section 6]{Ka13} that every OT manifold is isomorphic, as a complex manifold, to a solvmanifold -- i.e. to the quotient of a simply connected  solvable Lie group by a lattice -- with left invariant complex structure. It is thus natural to study the more general problem of LCP structures on solvmanifolds. This is the main topic of the present paper. Here is an outline of our results.

In Section 2 we recall the basics of Riemannian Lie groups, and some extensions to the conformal settings. In Section 3 we define LCP structures on Lie algebras and show that they correspond to left invariant LCP structures on the compact quotients of the corresponding simply connected Lie groups, whenever such compact quotients exist. 

A general construction method for unimodular LCP Lie algebras is given in Section 4, starting from a non-unimodular metric Lie algebra $\mh$ and an orthogonal representation of $\mh/[\mh,\mh]$. We also define amalgamated products of adapted LCP Lie algebras and show that they are again LCP.

In Section 5 we focus on the unimodular solvable case and show that every solvable LCP algebra is obtained by the previous construction. In particular, we obtain the classification of almost abelian unimodular LCP Lie algebras. 

In Section 6 we describe LCP algebras with flat space of codimension at most 3, and in Section 7 we obtain the full list of unimodular solvable Lie algebras of dimension up to 5 which carry LCP structures whose flat space is non-trivial and of positive codimension.

Finally, in Section 8 we study the problem of existence of lattices in the simply connected Lie groups corresponding to the LCP Lie algebras constructed above.

\smallskip 

{\bf Acknowledgements: } A.~M. and V.~dB. thank the French-Brazilian network in Mathematics and  the CAPES-COFECUB project 895/18 for financial support, and are grateful to the Mathematical Department of the Universitá degli Studi di Perugia, where this work was initiated. A.~A. would like to thank the Laboratoire de Math\'ematiques d'Orsay for hospitality. A.~A. is partially supported by CONICET and SECYT-UNC (Argentina).   V.~dB. is partially supported by FAPESP grant 2021/09197-8. The authors are supported by MATHAMSUD Regional Program 21-MATH-06.

\medskip 

\section{LCP structures on Riemannian Lie groups}

\subsection{LCP structures on compact Riemannian manifolds}

As explained in the Introduction, an LCP structure on a compact manifold $M$ is a conformal structure $c$ together with a closed non-exact Weyl structure $D$ with reducible holonomy. We will assume throughout the paper that the dimension of $M$ is at least 3.

To make things more precise, let us fix any Riemannian metric $g$ in the conformal class $c$. Then the Weyl structure $D$ can be written as $D=\nabla^\theta:=\nabla^g+\Theta$, where $\Theta$ is a $(2,1)$ tensor determined by some $1$-form $\theta$, called the Lee form of $D$ with respect to $g$, by the formula
\begin{equation}\label{eq:na}
    \Theta_XY:=\theta(X)Y+\theta(Y)X-g(X,Y)\theta^{\sharp_g},
\end{equation}
where $\theta^{\sharp_g}$ denotes the vector field which is the metric dual of $\theta$.
The Weyl structure is called closed or exact if $\theta$ is closed or exact respectively. Note that this does not depend on the choice of $g$ in the conformal class, since by the formulas of conformal change of the Levi-Civita connection (\cite[Theorem 1.159]{Besse2008}), the Lee form of $D$ with respect to any other metric $e^{2f}g\in c$ is $\theta-df$.
We can thus make the following alternative:

\begin{defi}\label{defi:lcp}
   An LCP structure on a compact manifold $M$ is a pair $(g,\theta)$ consisting in a Riemannian metric $g$ and a closed non-exact $1$-form $\theta$, called the Lee form,  such that the holonomy group of the connection $\nabla^\theta:=\nabla^g+\Theta$ -- where $\Theta$ is given by \eqref{eq:na} -- is reducible.
\end{defi}

 The advantage of this definition is that it can be stated only in terms of familiar objects (a Riemannian metric and a closed $1$-form). However, one should notice that the map $(g,\theta)\mapsto (c,D)$ given by $c:=[g]$ and $D:=\nabla^\theta$ is not one-to-one, since the conformal structure and Weyl connection defined by a pair $(g,\theta)$ are the same as those defined by $(e^{2f}g,\theta-df)$, for every smooth function $f$. We will however adopt this point of view here. 

Let $(g,\theta)$ be an LCP structure on $M$, and let $\pi:\tilde M\to M$ denote the universal cover. We consider the pull-back $(\tilde g:=\pi^*g,\,\tilde\theta:=\pi^*\theta)$  of the LCP structure to $\tilde M$. Since $\tilde M$ is simply connected, there exists a function $\varphi$ such that $\tilde\theta=d\varphi$. Using \cite[Theorem 1.159]{Besse2008} again, we see that the Levi-Civita connection of $h:=e^{2\varphi} \tilde g$ is given by 
$$\nabla^h=\nabla^{\tilde g}+\tilde\Theta\qquad\hbox{where}\qquad \tilde\Theta=\pi^*\Theta,$$
in other words, $\nabla^h$ is the lift of the Weyl connection $\nabla^\theta$ to $\tilde M$. 

Note that every $\gamma\in\pi_1(M)$ acts on $\tilde M$  homothetically with respect to $h$. Indeed, since $\gamma^*\tilde\theta=\tilde\theta$ we get $d(\gamma^*\varphi-\varphi)=0$, so there exists $c_\gamma\in\R$ such that $\gamma^*\varphi=\varphi+c_\gamma$, and therefore $\gamma^*h=e^{2c_\gamma}h$. Moreover, not all $c_\gamma$ vanish, since otherwise $\varphi$ would be the pull-back of a function on $M$, so $\theta$ would be exact.

By Definition \ref{defi:lcp}, $h$ has reducible holonomy. However, $(\tilde M,h)$ is incomplete since otherwise the elements $\gamma\in\pi_1(M)$ with $c_\gamma<0$ would have fixed points being contractions. We thus cannot apply de Rham's decomposition theorem, even though $\tilde M$ is simply connected. Nonetheless, we have the following striking result:

\begin{teo} {\rm (Kourganoff \cite[Theorem 1.5]{Kourganoff})}
  \label{kourg}
  Let $(\tilde M,\tilde g)$ be the universal cover of a compact LCP manifold $(M,g)$ and let $h=e^{2\varphi}\tilde g$ be the metric with reducible holonomy in the conformal class of $\tilde g$ introduced above. Then either $(\tilde M,h)$ is flat, or it is globally isometric to a Riemannian product $\mathbb{R}^q\times (N, g_N)$, where $\mathbb{R}^q$ is the flat Euclidean space, and $(N, g_N)$ is an incomplete Riemannian manifold with irreducible holonomy.
\end{teo}

An LCP structure $(g,\theta)$ on $M$ is called {\em adapted} if the lift of $\theta$ to the universal cover $\tilde M$ vanishes on the flat distribution $\R^q$ defined in the above result. Flamencourt proved that for every LCP structure $(g,\theta)$ on $M$, there exists $f\in C^\infty(M)$ such that $(e^{2f}g,\theta-df)$ is adapted \cite[Proposition 3.6]{Fl22}.
He used this result in order to construct LCP structures on every Riemannian product where one of the factors carries an adapted LCP structure \cite[Definition 3.9]{Fl22}.

Our main goal in this paper is to study LCP structures on compact Riemannian manifolds which can be written as quotients of Riemannian Lie groups by lattices (i.e. discrete co-compact subgroups). We give the necessary prerequisites in the next subsection.

 \subsection{Riemannian Lie groups}

A Riemannian Lie group is a Lie group $G$ endowed with a left invariant Riemannian metric $g$. The restriction of $g$ to the Lie algebra $\mg$ of $G$ is a scalar product, also denoted by $g$. 

A left invariant $1$-form on $G$ is equivalent to an element $\theta\in\mg^*$. The $1$-form $\theta$ is closed if and only if 
\begin{equation}\label{eq:tg}
    \theta|_{\mg'}=0,
\end{equation}
    where here, and throughout the paper, $\mg'$ denotes the commutator ideal $[\mg,\mg]$.

The Levi-Civita connection of $g$ induces a linear map $x\mapsto \nabla_x^g$, from $\mg$ to $\mathfrak{so}(\mg)$, defined by the Koszul formula
\begin{equation}
\label{eq:lc}\lela\nabla^g_xy,z\rira=\frac12\left(\lela  [x,y],z\rira-\lela[x,z],y\rira-\lela[y,z],x)\right)\qquad \forall\; x,y,z\in \mg.
\end{equation} Here and henceforth $\so(\mg)$ denotes the space of endomorphisms of $\mg$ which are skew-symmetric with respect to $g$. 

Recall that for any Lie algebra $\mg$, one can define a 1-form $ H^\mg\in\mg^*$ called the {\em trace form} given by
\begin{equation}
H^\mg(x):=\tr(\ad_x)\qquad \forall\; x\in\mg.
\end{equation}
For every $x,y\in \mg$ we have 
$$H^\mg([x,y])=\tr\ad_{[x,y]}=\tr[\ad_x,\ad_y]=0.$$
Thus $H^\mg(\mg')=0$, so $H^\mg$ is closed. 
By definition, $\mg$ is unimodular if and only if $H^\mg=0$.

Let $c$ denote the conformal structure on $G$ defined by the left invariant metric $g$. As explained above, every Weyl connection (i.e. torsion-free conformal connection) can be expressed as $\nabla^g+\Theta$, where $\Theta$ is defined by \eqref{eq:na}. If $\theta$ is left invariant, the Weyl connection is called left invariant. We then identify $\theta$ with the corresponding linear form $\theta\in\mg^*$. 

Any left invariant Weyl connection defines a linear map $\nabla^\theta:\mg\otimes\mg\to \mg$ satisfying
\begin{equation}
\label{eq:weylc}
\nabla^\theta_xy=\nabla^g_xy+\theta(x)y+\theta(y)x-g(x,y)\theta^\sharp\qquad \forall\; x,y\in \mg,
\end{equation}
where here and at several other places throughout the paper, 
$\theta^\sharp\in \mg$ denotes the $g$-dual vector in $\mg$.
In terms of the the skew-symmetric endomorphism $\theta\wedge x$ of $\mg$ defined by $(\theta\wedge x)(y):=\theta(y)x-g(x,y)\theta^\sharp$, this equation can also be written
\begin{equation}
\label{eq:weylc1}
\nabla^\theta_x=\nabla^g_x+\theta\wedge x+\theta(x)\Id_\mg\qquad\forall\; x\in \mg,
\end{equation}
showing that 
\begin{equation}
\label{eq:weylc2}
\nabla^\theta_x-\theta(x)\Id_\mg\in\so(\mg)\qquad\forall\; x\in \mg.
\end{equation}

Summarizing \eqref{eq:lc} and \eqref{eq:weylc}, the Weyl connection $\nabla^\theta$ is given by
    \begin{multline}
\label{eq:LCP}
\lela\nabla^\theta_xy,z\rira=\frac12\left(\lela  [x,y],z\rira-\lela[x,z],y\rira-\lela[y,z],x)\right)\\
+\theta(x)\lela y,z\rira+\theta(y)\lela x,z\rira-\theta(z)g(x,y)\qquad \forall\; x,y,z\in \mg.
\end{multline}

\subsection{Invariant LCP structures on Lie groups}
We will now introduce the notion of LCP structure on metric Lie algebras and prove that for every simply connected Lie group $G$ admitting lattices, there is a one-to-one correspondence between invariant LCP structures on quotients of $G$ by a lattice, and LCP structures on the (unimodular) Lie algebra of $G$.

Let $(\mg,g)$ be a metric Lie algebra, $\theta\in\mg^*$ a closed 1-form and $\nabla^\theta$ the connection defined in \eqref{eq:LCP}.
We denote by $R^\theta$ the curvature tensor of $\nabla^\theta$, that is 
\begin{equation}\label{eq:r}
R^\theta_{x,y}=[\nabla^\theta_x,\nabla^\theta_y]-\nabla^\theta_{[x,y]}\qquad\forall\; x,y\in\mg.
\end{equation}
We say that a subspace $\mmu\subset \mg$ is $\nabla^\theta$-parallel
if 
\begin{equation}
\label{eq:separa}
\nabla^\theta_x u\in \mmu \qquad \forall\;u \in \mmu,\,\forall\; x \in\mg,
\end{equation}
and $\nabla^\theta$-flat if it is $\nabla^\theta$-parallel and 
\begin{equation}
\label{eq:seflat}
R^\theta_{x,y}u=0 \qquad \forall\;u \in \mmu,\,\forall\; x,y \in\mg.
\end{equation}

\begin{defi} \label{defi} A locally conformally product (LCP) Lie algebra is a quadruple $(\mg,g,\theta,\mmu)$ where $(\mg,g)$ is a metric Lie algebra, $\theta$ is  a non-zero closed $1$-form on $\mg^*$ and $\mmu$ is a $\nabla^\theta$-flat subspace of $(\mg,g)$. The LCP structure $(g,\theta,\mmu)$ on $\mg$ is called {\em adapted} if $\theta|_\mmu=0$, {\em degenerate} if $\mmu=0$, and {\em conformally flat} if $\mmu=\mg$.
\end{defi}

Since $\nabla^\theta$ does not change if the metric $g$ is replaced by $\lambda g$ for some $\lambda>0$, it follows that the LCP condition is invariant to constant rescalings of the metric.

Note that the case of conformally flat Lie algebras has been studied by Maier \cite{Maier98}.

As already mentioned, our main objective will be to study LCP structures on compact Riemannian manifolds obtained as quotients of Riemannian Lie groups by lattices. 

Let $G$ be a simply connected Lie group with Lie algebra $\mg$ and let $\Gamma\subset G$ be any lattice. It is well known that if $G$ has a lattice then its Lie algebra $\mg$ is unimodular \cite{Mil76}. 

We consider the quotient $M:=\Gamma\backslash G$. Recall that every element in the tensor algebra of $\mg$ defines a left invariant tensor on $G$, which is in particular $\Gamma$-invariant, so projects to a tensor on $M$. From now on, let $\theta\in \mg^*$ and $g\in\Sym^2(\mg^*)$ be a closed 1-form and a scalar product. We will denote by the same letters the corresponding closed 1-form and Riemannian metric on $G$ and on $M$. We then have:

\begin{pro}\label{pro:equiv}
    The pair $(g,\theta)\in \Sym^2(T^*M)\times \Omega^1(M)$ on $M$ induced by a scalar product $g\in\Sym^2(\mg^*)$ and a closed non-zero $1$-form $\mg^*$ is an LCP structure if and only if there exists a vector subspace $\mmu\subset\mg$ such that $(g,\theta,\mmu)$ is an LCP structure on $\mg$.
\end{pro}
\begin{proof}
    Assume first that $(g,\theta)$ is an LCP structure on $M$. Then its pull-back to the universal cover $G$ is a left invariant LCP structure on $G$. By Theorem \ref{kourg}, there exists a $\nabla^\theta$-flat distribution $U$ on the Riemannian Lie group $(G,g)$. As the Weyl connection $\nabla^\theta$ is left invariant, $U$ is left-invariant too, so defines a subspace $\mmu$ of $\mg$ which is $\nabla^\theta$-flat.

    Conversely, if $(g,\theta,\mmu)$ is an LCP structure on $\mg$, the $\nabla^\theta$-flat subspace $\mmu$ of $\mg$ determines by left translations a left invariant distribution $U$ on $G$ which is parallel with respect to the connection $\nabla^\theta$. This shows that the corresponding Weyl connection $\nabla^\theta$ on the quotient $M=\Gamma\backslash G$ has reducible holonomy, so $(M,g,\theta)$ is LCP.
\end{proof}

\medskip 

\section{LCP structures on Lie algebras}

In view of Proposition \ref{pro:equiv}, we will from now on study LCP structures on Lie algebras. Then, in the last section, we will investigate the existence of lattices in the simply connected groups corresponding to LCP Lie algebras. We start with some general results.

\begin{lm}\label{lm:sub}
    Let $(\mg,g,\theta,\mmu)$ be an LCP  Lie algebra.
    Then the subspaces $\mmu$ and its $g$-orthogonal complement $\mmu^\bot$ are subalgebras of $\mg$.
\end{lm}
\begin{proof} From \eqref{eq:lc} and \eqref{eq:weylc} we immediately obtain
\begin{equation}
\label{eq:torsion}
 [x,y]=   \nabla^\theta_xy-\nabla^\theta_yx\qquad \forall\; x,y\in \mg,
\end{equation}
which of course just translates the fact that $\nabla^\theta$ is torsion-free.
By \eqref{eq:torsion}, every $\nabla^\theta$-parallel space is a subalgebra of $\mg$. Moreover, \eqref{eq:weylc1} shows that the orthogonal complement of a $\nabla^\theta$-parallel space is again $\nabla^\theta$-parallel.  
\end{proof}

We can now give a more explicit characterisation of LCP structures:

\begin{pro}\label{pro:algLCP} 
Let $(\mg,g)$ be a metric Lie algebra, $\theta\in\mg^*$ a non-zero closed $1$-form, and $\mmu\subset\mg$ a vector subspace.  Then $\mmu$ is $\nabla^\theta$-parallel if and only if the following two conditions hold:
\begin{enumerate}
\item $\mmu$ and $\mmu^\perp$ are Lie subalgebras;
\item for every $u\in\mmu$ and $x\in\mmu^\perp$, 
\begin{equation}\label{eq:xu}g([u,x],x)=\theta(u)|x|^2\qquad\hbox{and}\qquad g([x,u],u)=\theta(x)|u|^2.
\end{equation}
\end{enumerate}
Moreover, $(\mg,g,\theta,\mmu)$ is LCP if and only if in addition to $(1)$ and $(2)$, the following condition holds:
\begin{enumerate}
\item[(3)] the map $\nabla^\theta:\mg\to \so(\mmu)\oplus\R\Id_\mmu\subset \gl(\mmu)$, defined by $x\mapsto \nabla_x^\theta|_\mmu$, is a Lie algebra representation.
\end{enumerate}
\end{pro}
\begin{proof}
    Assume that $\mmu$ is $\nabla^\theta$-parallel. Then (1) follows from Lemma \ref{lm:sub}. 
    
Taking $x=z\in\mmu^\bot$ and $y\in\mmu$ in \eqref{eq:LCP} and using that $\lela\nabla^\theta_xu,x\rira=0$ since $\mmu$ is $\nabla^\theta$-parallel, yields the first part of \eqref{eq:xu}.
Similarly, taking $x=z\in\mmu$ and $y\in\mmu^\bot$ in \eqref{eq:LCP} yields the second part of \eqref{eq:xu}.

If moreover $(g,\theta,\mmu)$ is LCP, then $\mmu$ is $\nabla^\theta$-flat, so (3) follows from \eqref{eq:r}.

Conversely, assume that (1)--(2) hold. We need to show that $\lela\nabla^\theta_xy,z\rira$ vanishes whenever $y\in \mmu$ and $z\in\mmu^\perp$. If $x\in\mmu$, using (1) and \eqref{eq:LCP} we get
\begin{eqnarray*}
    \lela\nabla^\theta_xy,z\rira&=&\frac12\left(-\lela[x,z],y\rira-\lela[y,z],x)\right)-\theta(z)g(x,y)\\
    &=&\frac12\left(\lela[z,x],y\rira+\lela[z,y],x)\right)-\theta(z)g(x,y)=0
\end{eqnarray*}
    by polarising the second part of \eqref{eq:xu}. Similarly, for $x\in \mmu^\perp$, by (1) and \eqref{eq:LCP} we obtain
    \begin{eqnarray*}
    \lela\nabla^\theta_xy,z\rira&=&\frac12\left(\lela[x,y],z\rira-\lela[y,z],x)\right)+\theta(y)g(x,z)\\
    &=&\frac12\left(-\lela[y,x],z\rira-\lela[y,z],x)\right)-\theta(y)g(x,z)=0
\end{eqnarray*}
by polarising the first part of \eqref{eq:xu}. Thus $\mmu$ is $\nabla^\theta$-parallel. 

If, in addition, (3) holds, the fact that $\mmu$ is $\nabla^\theta$-flat follows directly from \eqref{eq:r}.
\end{proof}

\begin{cor} \label{cor:3.4}
Let $(\mg,g,\theta,\mmu)$ be an LCP  Lie algebra and denote by $q$ and $n$ the dimensions of $\mmu$ and $\mg$ respectively. If $\mg$ is unimodular, the trace forms of $\mmu$ and $\mmu^\bot$ are related to $\theta$ by the relations
    \begin{equation}\label{eq:trace}
        H^\mmu=-(n-q) \,\theta|_\mmu,\qquad H^{\mmu^\bot}=-q\,\theta|_{\mmu^\bot}.
    \end{equation}
\end{cor}

\begin{proof}
    Let $(e_i)_{1\le i\le q}$ and $(f_j)_{1\le j\le n-q}$ be some orthonormal bases of $\mmu$ and $\mmu^\bot$. Since $\mg$ is assumed to be unimodular, for every $y\in \mg$ we have 
$$0=\tr(\ad_y)=\sum_{i=1}^qg([y,e_i],e_i)+\sum_{j=1}^{n-q}g([y,f_j],f_j),$$
    whence
\begin{equation}\label{eq:hU}
H^\mmu(y)+\sum_{j=1}^{n-q}g([y,f_j],f_j)=0\qquad\forall\; y\in \mmu,
\end{equation}
and
\begin{equation}\label{eq:hUbot}
\sum_{i=1}^qg([y,e_i],e_i)+H^{\mmu^\bot}(y)=0\qquad\forall\; y\in \mmu^\bot.
\end{equation}
Then \eqref{eq:trace} follows immediately by taking $x=f_j$ in the first part of \eqref{eq:xu} and $u=e_i$ in the second part of \eqref{eq:xu}.
\end{proof}

\begin{cor}\label{cor:abelian} On an abelian Lie algebra of dimension $n\ge 3$, every LCP structure is degenerate.
\end{cor}
\begin{proof}
Assume that $(\mg,g,\theta,\mmu)$ is an LCP abelian Lie algebra. By Proposition \ref{pro:algLCP} (2) and the fact that $\theta\ne 0$, we see that either $\mmu=0$, so the LCP structure is degenerate, or $\mmu^\perp=0$, so $(\mg,g,\theta)$ is conformally flat. 

In the latter case, \eqref{eq:weylc1} together with Proposition \ref{pro:algLCP} (2) show that 
$$0=[\nabla^\theta_x,\nabla^\theta_y]=[\theta\wedge x,\theta\wedge y]\qquad \forall\; x,y\in\mg.$$
In particular, for every $x,y$ non-zero vectors in $\ker(\theta)$, this yields $0=|\theta|^2x\wedge y$, so $\ker(\theta)$ is 1-dimensional, which contradicts the fact that $n\ge 3$.
\end{proof}

\begin{cor}\label{cor:modif} If $(\mg,g,\theta,\mmu)$ is an adapted LCP Lie algebra, then for every $\lambda\in\R$ such that $g_\lambda:=g+\lambda\theta\otimes\theta$ is positive definite, $(\mg,g_\lambda,\theta,\mmu)$ is also an adapted LCP Lie algebra.
\end{cor}
\begin{proof} Since $(\mg,g,\theta,\mmu)$ is adapted (i.e. $\theta|_\mmu=0)$), the $g$-orthogonal of $\mmu$ coincides with the $g_\lambda$-orthogonal of $\mmu$. Condition (1) in Proposition \ref{pro:algLCP} is thus verified for $g_\lambda$. Next, for $u\in\mmu$ and $x\in\mmu^\perp$, we have by \eqref{eq:xu} and \eqref{eq:tg}
$$g_\lambda([u,x],x)=g([u,x],x)+\theta([u,x])\theta(x)=\theta(u)|x|^2_g=0=\theta(u)|x|^2_{g_\lambda}$$
and 
$$g_\lambda([x,u],u)=g([x,u],u)+\theta([x,u])\theta(u)=\theta(x)|u|^2_g=\theta(x)|u|^2_{g_\lambda},$$
thus showing that Condition (2) in Proposition \ref{pro:algLCP} holds for $g_\lambda$. 

Finally, in order to check Condition (3) in Proposition \ref{pro:algLCP}, let us denote by $\nabla^\lambda$ the Weyl connection $\nabla^{g_\lambda}+\Theta$ on $(\mg,g_\lambda)$, where $\Theta$ is defined by \eqref{eq:na}. Since $g$ and $g_\lambda$ coincide when one of the arguments is in $\mmu$ or in $\mg'$, \eqref{eq:LCP} applied to $z\in \mmu$ immediately yields 
\begin{equation}\label{eq:gl}g(\nabla^\lambda_xy,z)=g_\lambda(\nabla^\lambda_xy,z)=g(\nabla^\theta_xy,z).\end{equation}
This shows that $\mmu^\perp$ is invariant by $\nabla^\lambda_x$ for every $x\in\mg$, so $\mmu$ is invariant too, as $\nabla^\lambda$ is a Weyl structure. Moreover, \eqref{eq:gl} shows that the restriction of $\nabla^\lambda_x$ to $\mmu$ is equal to the restriction of $\nabla^\theta_x$ to $\mmu$ for every $x\in\mg$, so Condition (3) in Proposition \ref{pro:algLCP} holds.
\end{proof}

\medskip

\section{Constructions of LCP structures}

In this section we introduce two different methods to construct LCP structures on Lie algebras. Whilst the first one builds on pre-existing LCP structures, the second one defines  LCP structures on Lie algebras obtained as semidirect products.

\subsection{Products of adapted LCP Lie algebras}

Recall (Definition \ref{defi}) that an LCP Lie algebra $(\mg,g,\theta,\mmu)$ is called adapted if $\theta|_\mmu=0$. In this setting, the extension procedure of Flamencourt \cite[Definition 3.9]{Fl22} can be stated as follows:

\begin{pro} \label{pro:direct}If $(\mg,g,\theta,\mmu)$ is an adapted LCP  Lie algebra and $(\mk,k)$ is an arbitrary metric Lie algebra, then $(\tilde \mg:=\mg\oplus\mk,\,\tilde g:=g+k,\,\tilde\theta:=\theta,\,\tilde\mmu:=\mmu)$ is again an LCP Lie algebra.
\end{pro}

\begin{proof} We will check that $(\tilde \mg,\tilde g,\tilde\theta,\tilde\mmu)$ satisfies the three conditions in Proposition \ref{pro:algLCP}, using the fact that $(\mg,g,\theta,\mmu)$ satisfies them.

By Condition (1), $\mmu$ and $\mmu^\perp$ are subalgebras of $\mg$, so $\mmu$ and its orthogonal complement $\tilde\mmu^\perp:=\mmu^\perp\oplus\mk$ in $\tilde\mg$ are subalgebras of $\tilde\mg$. 

Equation \eqref{eq:xu} holds for every $u\in\mmu$ and $x\in\mmu^\perp$. Moreover, since the LCP structure is adapted, $\theta|_\mmu=0$, so the first part in \eqref{eq:xu} reads $g([u, x],x)=0$. For every $\tilde x$ in $\tilde\mmu^\bot$ written as $\tilde x=x+y$ with $x\in\mmu^\perp$ and $y\in\mk$ we have $[u,\tilde x]=[u,x]$, so
$$\tilde g([u,\tilde x],\tilde x)=\tilde g([u, x],\tilde x)=g([u, x],x)=0=\theta(u)|\tilde x|_{\tilde g}^2.$$
Moreover, from the second part of \eqref{eq:xu} we get
$$\tilde g([\tilde x,u],u)=g([x,u],u)=\theta(x)|u|^2=\tilde\theta(\tilde x)|u|_{\tilde g}^2,$$
showing that $(\tilde \mg,\tilde g,\tilde\theta,\tilde\mmu)$ satisfies Condition (2).

Finally, applying \eqref{eq:LCP} to the Weyl connection $\nabla^{\tilde\theta}$ on the metric algebra $(\tilde\mg,\tilde g)$, we get $\tilde g(\nabla^{\tilde\theta}_xy,z)=0$ for every $x\in \mk$ and $y,z\in\mmu$. This shows that the map $\nabla^{\tilde\theta}:\tilde\mg\to\gl(\tilde \mmu)$ is the extension by $0$ on $\mk$ of the map $\nabla^{\theta}:\mg\to\gl(\mmu)$, so it is a Lie algebra representation.
\end{proof}

  In the context of Lie algebras, this construction can be generalised as follows.

Let $(\mg_i,g_i,\theta_i,\mmu_i)$, $i=1,2$, be adapted LCP Lie algebras. Consider the direct sum of these Lie algebras $\mg_1\oplus\mg_2$ endowed with the inner product $g_1+g_2$.

The 1-forms $\theta_1$, $\theta_2$ extend in an obvious way to closed 1-forms on  $\mg_1\oplus\mg_2$.  It is easy to check that $\mg:=\ker (\theta_1-\theta_2)$ is a codimension one ideal of $\mg_1\oplus\mg_2$. Indeed,  $\theta_i(\mg_1'\oplus\mg_2')=\theta_i((\mg_1\oplus\mg_2)')=0$ for $i=1,2$, since both $\theta_1$ and $\theta_2$ are closed, which implies $\mg_1'\oplus\mg_2'\subset \mg$. We denote by $g:=(g_1+g_2)|_\mg$. In addition, $\theta:=\theta_1|_\mg=\theta_2|_\mg$ is a non-zero 1-form in $\mg$, which is closed because $\theta(\mg_1'\oplus\mg_2')=0$. Finally, the fact that both LCP Lie algebras are adapted implies $\mmu:=\mmu_1\oplus \mmu_2\subset \mg$. Note that $\theta|_\mmu=0$, since if $u_i\in\mmu_i$, $i=1,2$, then $\theta(u_i)=\theta_i(u_i)=0$ due to the definition of $\theta$ and the fact that both LCP Lie algebras are adapted.

\begin{pro}\label{pro:amal} For every adapted LCP Lie algebras $(\mg_i,g_i,\theta_i,\mmu_i)$, $i=1,2$, the quadruple $(\mg,g,\theta,\mmu)$ defined above is an adapted LCP Lie algebra.
Moreover, if $\mg_i$ is unimodular (respectively solvable) for $i=1,2$, then so is $\mg$.
\end{pro}

\begin{proof}
We shall use Proposition \ref{pro:algLCP}. Since $\mmu_i$ is a subalgebra of $\mg_i$ we have that $\mmu=\mmu_1\oplus \mmu_2$ is a subalgebra of $\mg_1\oplus \mg_2$. Moreover, as $\mmu\subset\mg$, it is clear that $\mmu$ is a subalgebra of $\mg$. On the other hand, let us denote $\mmu^\perp$ the orthogonal complement of $\mmu$ in $\mg$. Then $\mmu^\perp =(\mmu_1^\perp \oplus \mmu_2^\perp)\cap \mg$, which is a subalgebra of $\mg$ since $\mmu_i^\perp$ is a subalgebra of $\mg_i$ for each $i$. Hence, Condition (1) in Proposition \ref{pro:algLCP} holds for $(\mg,g,\theta,\mmu)$. 

Next, let $u=u_1+u_2\in\mmu$ and $x=x_1+x_2\in\mg$, where each subscript indicates the component in $\mg_1$ and $\mg_2$, respectively. Then:
\begin{align*}
g([u,x],x) & = g([u_1+u_2,x_1+x_2],x_1+x_2) \\
& = g([u_1,x_1],x_1)+g([u_2,x_2],x_2)\\
& = \theta_1(u_1)|x_1|^2+\theta_2(u_2)|x_2|^2 \qquad \text{(by Proposition \ref{pro:algLCP}(2))}\\
& = 0 \qquad (\text{since}\ \theta_i|_{\mmu_i}=0)\\
& =\theta(u)|x|^2 \qquad (\text{since}\ \theta|_\mmu=0),
\end{align*}
and 
\begin{align*}
g([x,u],u) & = g([x_1+x_2,u_1+u_2],u_1+u_2) \\
& = g([x_1,u_1],u_1)+g([x_2,u_2],u_2)\\
& = \theta_1(x_1)|u_1|^2+\theta_2(x_2)|u_2|^2 \qquad \text{(by Proposition \ref{pro:algLCP}(2))}\\
& = \theta(x)|u|^2 \qquad (\text{since}\ \theta(x)=\theta_1(x_1)=\theta_2(x_2)).
\end{align*}
Thus, Condition (2) in Proposition \ref{pro:algLCP} holds for $(\mg,g,\theta,\mmu)$.

Finally, we check Condition (3) in Proposition \ref{pro:algLCP}. 
For every $x_1+x_2\in\mg$, $u_1,v_1\in\mmu_1$ and $u_2,v_2\in\mmu_2$ we have $\theta(x_1+x_2)=\theta_1(x_1)=\theta_2(x_2)$ and $\theta_i(u_i)=\theta_i(v_i)=0$. Therefore \eqref{eq:LCP} shows immediately that
\[ (g_1+g_2)(\nabla^\theta_{x_1+x_2}(u_1+u_2),v_1+v_2)=g_1(\nabla^{\theta_1}_{x_1}u_1,v_1)+g_2(\nabla^{\theta_2}_{x_2}u_2,v_2),\]
and this shows that $\nabla^\theta$ is the restriction to $\mg$ of the representation 
\[ \nabla^{\theta_1}+\nabla^{\theta_2}:\mg_1\oplus\mg_2\to \gl(\mmu_1)\oplus\gl(\mmu_2)\subset \gl(\mmu).\] 
Thus, we have proved that $(\mg,g,\theta,\mmu)$ is an LCP Lie algebra.

The last assertions follow immediately from the fact that if both $\mg_1$ and $\mg_2$ are unimodular or solvable then $\mg_1\oplus \mg_2$ has the same property. Solvability is inherited by subalgebras, and since the derived algebra of $\mg_1\oplus \mg_2$ is contained in $\mg$, the latter inherits the unimodularity property.
\end{proof}

The LCP Lie algebra constructed in Proposition \ref{pro:amal} is called the {\em amalgamated product} of the adapted LCP Lie algebras $(\mg_1,g_1,\theta_1,\mmu_1)$ and $(\mg_2,g_2,\theta_2,\mmu_2)$.

It turns out that the direct product of an LCP Lie algebra with an arbitrary metric Lie algebra (which is LCP by Proposition \ref{pro:direct}) is in some sense a special case of amalgamated product. 

To see this, let $(\mg_1,g_1,\theta_1,\mmu_1)$ be an adapted LCP Lie algebra and let $(\mh,h)$ be any metric Lie algebra. Consider $\mg_2:=\R b \oplus\mh$ with the scalar product $g_2$ which extends $h$ and makes $b\bot\mh$ and $b$ of unit norm. Then $(\mg_2,g_2,\theta_2,\mmu_2)$ is a degenerate (and thus adapted) LCP Lie algebra, with respect to the closed 1-form $\theta_2:=g_2(b,\cdot)$  and $\mmu_2:=0$.

Let $(\mg,g,\theta,\mmu)$ be the amalgamated product of $(\mg_1,g_1,\theta_1,\mmu_1)$ and $(\mg_2,g_2,\theta_2,\mmu_2)$.
We define $f:\mg_1\oplus\mh\to \mg$ by
$$f((x_1,x_2)):=x_1+\theta_1(x_1)b+x_2.$$
This is a Lie algebra isomorphism (since $\theta_1$ is closed), but not an isometry, since $f^*g=(g_1+\theta_1\otimes\theta_1)+h$. This means that the direct product LCP structure $(g_1+h,\theta_1,\mmu_1)$ on $\mg_1\oplus \mh$ coincides with the pull-back by $f$ of the amalgamate LCP structure on $\mg$ up to adding an extra factor $\theta_1\otimes\theta_1$ to the metric, which by Corollary \ref{cor:modif} does not change the property of being LCP.

\subsection{LCP structures on semidirect products.}

Motivated by a result which will be proved in the next section, we give here a general construction procedure for LCP Lie algebras.

\begin{pro}\label{pro:cnstr}
Let $(\mh,h)$ be a non-unimodular metric Lie algebra with trace form $H^\mh\in \mh^*$ and let $q\ge 1$ be an integer. Assume that $\beta: \mh\to \so(\R^q)$ is a Lie algebra representation such that $\beta|_{\mh'}=0$. 

Then $\alpha: \mh\to \gl(\R^q)$ defined by $\alpha(x):=-\frac1qH^\mh(x)\Id_{\R^q}+\beta(x)$ for $x\in\mh$ is a Lie algebra representation, and $(\mg,g,\theta,\mmu)$ is a non-degenerate unimodular LCP Lie algebra for $\mg:= \mh\ltimes_{\alpha}\R^q,\ g:=\langle\cdot,\cdot\rangle_{\R^q}+ h,\ \theta:=-\frac1qH^\mh$ (extended by $0$ on $\R^q$), and $\mmu:=\R^q$.

Moreover, $\mg$ is solvable if and only if $\mh$ is solvable.
\end{pro}

\begin{proof} 
Since $\beta$ as well as the closed 1-form $\theta$ vanish on $\mh'$, we get 
\begin{equation}\label{eq:gg}
    \alpha(\mg')=0.
\end{equation}
Thus for every $x\in\mg$ we can write
$$[\alpha(x),\alpha(y)]=[\beta(x),\beta(y)]=\beta([x,y])=0=\alpha([x,y]).$$
This shows that $\alpha$ is a Lie algebra representation, so $\mg:=\mh\ltimes_{\alpha} \R^q$ is well defined.

By the definition of a semidirect product, the bracket on $\mg$ satisfies
\begin{equation}
   \ad_x|_{\R^q}=\alpha(x)=-\frac1qH^\mh(x)\Id_{\R^q}+\beta(x)\qquad \forall\; x\in\mh.\label{eq:xht}
\end{equation}

For every $u\in\R^q$ we clearly have $\tr(\ad_u)=0$ since $\ad_u(\R^q)=0$ and $\ad_u(\mh)\subset \R^q$. For $x\in \mh$ we compute using \eqref{eq:xht} and the fact that $\beta(x)\in\so(q)$:
$$\tr\ad_x=\tr(\ad_x|_\mh)+\tr(\ad_x|_{\R^q})=H^\mh(x)+\tr(-\frac1qH^\mh(x)\Id_{\R^q}+\beta(x))=0.$$
Thus $\mg$ is unimodular.

In order to prove that $(\mg,g,\theta,\mmu)$ is an LCP Lie algebra, we need to check the conditions (1)--(3) from Proposition \ref{pro:algLCP}. 

Since $\mmu=\R^k$ is an ideal of $\mg$ and $\mmu^\perp=\mh$ is a subalgebra, (1) is clearly verified. 

For $u\in\mmu$ and $x\in\mmu^\perp$ we have $[u,x]\in\mmu$, $\theta(u)=0$, and 
$$g([x,u],u)=g(\alpha(x)(u),u)=g(\theta(x)u+\beta(x)(u),u)=\theta(x)|u|^2,$$
(since $\beta(x)\in\so(q)$), thus proving (2). 

We now claim that
\begin{equation}\label{eq:nab}
g(\nabla^\theta_xy,z)=g([x,y],z)\qquad \forall\; y,z\in\mmu, \, x\in\mg.
\end{equation}
For $x\in\mmu$ both terms vanish because $\mmu$ is an abelian ideal of $\mg$ and $\theta|_\mmu=0$. For $x\in\mmu^\perp$ we have by \eqref{eq:LCP}
\begin{eqnarray*}
g(\nabla^\theta_xy,z)&=&\frac12(g(\alpha(x)(y),z)-g(\alpha(x)(z),y))+\theta(x)g(y,z)\\&=&g(\beta(x)(y),z)+\theta(x)g(y,z)\\&=&g([x,y],z).
\end{eqnarray*}

On the other hand, from the proof of Proposition \ref{pro:algLCP}, (1) and (2) show that $\mmu$ is $\nabla^\theta$-parallel, and since $\mmu$ is an ideal of $\mg$, \eqref{eq:nab} implies $\nabla^\theta_x|_\mmu=\ad_x|_\mmu$ for every $x\in \mg$, so (3) follows directly from the Jacobi identity.

For the last statement we notice that $\mg'\subset \R^q\oplus\mh'$ and by \eqref{eq:gg} we have $[\mg',\mg']\subset[\mh',\mh']$. Thus if $\mh$ is solvable, $\mg$ is solvable as well. Conversely, since $\mh$ is a Lie subalgebra of $\mg$, it is solvable whenever $\mg$ is solvable.
\end{proof}

\begin{remark}
 The construction above shows that any non-unimodular Lie algebra  $\mh$ is a codimension 1 subalgebra of an LCP Lie algebra. Indeed, it suffices to take $q=1$ and $\beta=0$ in Proposition \ref{pro:cnstr}.
\end{remark} 

The proposition above gives a criterion for the existence of non-degenerate LCP structures on a given unimodular Lie algebra. 

\begin{cor}\label{cor:converse}Let $\mg$ be an unimodular Lie algebra carrying a non-unimodular subalgebra $\mh$ and an abelian ideal $\mmu$ of dimension $q\ge 1$ such that $\mg=\mh\oplus \mmu$. Consider the closed $1$-form $\theta:=-\frac1{q}H^\mh$, extended to $\mg$ by $0$ on $\mmu$. Assume that $\ad_x|_\mmu=0$ for every $x\in\mh'$ and that there exists a scalar product $g_\mmu$ on $\mmu$ such that $\ad_x|_\mmu-\theta(x)\Id_\mmu\in\so(\mmu)$ for every $x\in \mh$. Then for every scalar product $g_\mh$ on $\mh$, $(\mg,g_\mh+g_\mmu,\theta,\mmu)$ is a non-degenerate LCP Lie algebra. 
\end{cor}
\begin{proof}
By definition, we can write $\mg=\mh\ltimes_\alpha\mmu$, where $\alpha(x):=\ad_x|_\mmu$ for every $x\in\mh$. By assumption, there is a map $\mh\ni x\mapsto \beta(x):=\alpha(x)-\theta(x)\Id_\mmu\in\so(\mmu)$. We claim that $\beta$ is a Lie algebra representation vanishing on $\mh'$. Indeed, since $\theta$ is closed, it vanishes on $\mh'$, so for every $x,y\in\mh$ we have 
$$[\beta(x),\beta(y)]=[\alpha(x),\alpha(y)]=[\ad_x|_\mmu,\ad_y|_\mmu]=\ad_{[x,y]}|_\mmu=0=\beta([x,y]).$$
The statement thus follows directly from Proposition \ref{pro:cnstr}.
\end{proof}

In the remaining part of this section we will use Proposition \ref{pro:cnstr} in order to construct some explicit examples of LCP Lie algebras.

\begin{ex}\label{ex:almab}
Let  $A\in \gl(p)$ with $\tr(A)\neq 0$, $B\in \so(q)$, and set $\mh:=\R b\ltimes_A \R^p$ where, by a slight abuse of notation, $A$ denotes the representation of $\R  b$ on $\R^p$ sending $b$ to $A$; notice that $\mh$ is not unimodular but solvable. The standard inner product $\langle\cdot,\cdot\rangle_{\R^p}$ on $\R^p$ extends to an inner product $h$ on $\mh$ such that $h(b,\R^p)=0$ and $h(b,b)=1$. It is easy to check that $H^\mh(\cdot)=\tr(A)\, h(b,\cdot)$.

The linear map $\beta:\mh\to\so(q)$ defined by $\beta(b)=B$ and $\beta(\R^p)=0$, is clearly a Lie algebra representation vanishing on $\mh'$. 
According to Proposition \ref{pro:cnstr}, the Lie algebra $\mg:=\mh\ltimes_\alpha\R^q$ where $\alpha(x):=-\frac1q\tr(A) \,h(b,x)\Id_{\R^q}+\beta(x)$, for $x\in\mh$, is unimodular solvable  and $$\mathcal L_{(A,B)}:=(\mg,g,-\frac1q \tr( A)\, h(b,\cdot),\R^q)$$ is an LCP Lie algebra for $g:=h+\langle\cdot,\cdot\rangle_{\R^q}$. The Lie algebra $\mg$ is actually almost abelian, since $\R^p\oplus \R^q$ is an abelian ideal of codimension 1. Notice that the action of $b$ on this ideal is just
$$\ad_b|_{\R^p\oplus\R^q}=\left[\begin{array}{c|c}  
  A&\\ \hline 
 & B-\frac1q\tr( A)\Id_{\R^q} 
\end{array}\right].$$
\end{ex}

A slight modification of the previous example gives LCP structures on unimodular solvable Lie algebras which are not almost abelian.

\begin{ex}\label{ex:noalmab} 
Let $A\in \gl(p)$ with $\tr( A)\neq 0$, let $B_1,B_2\in \so(q)$ be  endomorphisms such that $B_2\neq 0$ and $[B_1,B_2]=0$, and $v\in\R^p$. Set $\mh:=\R b\ltimes \R^{p+1}$ where $\R^{p+1}=\R y\oplus \R^p$ and the semidirect product law satisfies 
\[
(\ad_b)|_{\R^p}=A, \quad [b,y]=v\in\R^p;
\] 
again $\mh$ is solvable and non-unimodular. Consider an inner product $h$ on $\mh$ such that $h(b,\R^p)=h(y,\R^p)=h(y,b)=0$ and $h(b,b)=h(y,y)=1$. As before, we have $H^\mh(\cdot)=\tr (A)\, h(b,\cdot)$.

The linear map $\beta:\mh\to\so(q)$ satisfying $\beta(b)=B_1$, $\beta(y)=B_2$ and $\beta(\R^p)=0$ is a Lie algebra representation vanishing on $\mh'$.

By Proposition \ref{pro:cnstr}, the Lie algebra $\mg=\mh\ltimes_\alpha \R^q$ with $\alpha(x):=-\frac1q\tr (A) \,h(b,x)\Id_{\R^q}+\beta(x)$ for every $x\in\mh$, is unimodular and solvable. Moreover, if  $g$ is the scalar product extending $h$ and making $\R^q\bot\mh$, then $$\mathcal C_{(A,B_1,B_2,v)}:=(\mg,g,-\frac1q \tr (A)\, h(b,\cdot),\R^q)$$ is an LCP Lie algebra.

Explicitly, the actions of $b$ and $y$ on $\mg$ satisfy
\begin{eqnarray*}
    &(\ad_b)|_{\R^p}=A, \quad [b,y]=v\in\R^p, \quad (\ad_b)|_{\R^q}=B_1-\frac1q\tr(A) \Id_{\R^q} \\
    &(\ad_y)|_{\R^p}=0, \quad [y,b]=-v\in\R^p, \quad (\ad_y)|_{\R^q}=B_2.\end{eqnarray*}

Notice that $\mg$ is not almost abelian. Indeed, assume $\mk$ is an abelian ideal of $\mg$ of codimension 1 and write $\mg=\R \xi\oplus \mk$ for some $\xi\in\mg$, then, $\mg'=[\R \xi\oplus \mk,\R \xi\oplus\mk]=[\R \xi,\mk]\subset \mk$. 

By definition of $\mg$, $[y,\R^p]=[\R^p\oplus\R^q,\R^p]=[\R^p\oplus\R^q,\R^q]=0$. So if an element of the form $\lambda b+\mu y+x$ is in $\mk$, for some $\lambda,\mu\in\R$, $x\in \R^{p+q}$, then for any $u\in\R^p$ such that $Au\neq 0$, we have $[b,u]=Au\in\mg'\subset \mk$ and thus
\[
0=[\lambda b+\mu y+x,u]=\lambda Au \Rightarrow \lambda=0.
\]Moreover, given $w\in\R^q$ such that $[y,w]=A_2w\neq 0$,
\[
0=[\mu y+x,w]=\mu B_2w\Rightarrow \mu=0.
\] Therefore, $\mk\subset \R^p\oplus \R^q$ contradicting the codimension 1 hypothesis.

However, it is straightforward to check that $\R y\oplus \R^p\oplus \R^q$ is an almost abelian ideal of codimension 1 of $\mg$. Note that this ideal is not nilpotent, since $B_2$ is a non-zero skew-symmetric endomorphism.
\end{ex}

\section{LCP structures on unimodular solvable Lie algebras}

In this section we focus on the structure of unimodular solvable Lie algebras carrying LCP structures. We show that each such Lie algebra is a semidirect product of a non-unimodular Lie subalgebra, acting on the (abelian) flat factor of the LCP structure, as in Proposition \ref{pro:cnstr}. 

In this way, we establish a one-to-one correspondence between unimodular solvable LCP Lie algebras, and triples $(\mh,h,\beta)$, where $(\mh,h)$ is a non-unimodular solvable metric Lie algebra, and $\beta:\mh\to\so(q)$ is an orthogonal Lie algebra representation vanishing on $\mh'$.

We start with the following:

\begin{lm}\label{lm:beta0} Let  $(\mg,g,\theta,\mmu)$ be a non-degenerate solvable LCP Lie algebra. Then  $\nabla^\theta_xy=0$ for all $x\in\mg'$ and $y\in\mmu$.
\end{lm}

\begin{proof}
Since $\nabla^\theta:\mg \to\so( \mmu)\oplus\R$ is a Lie algebra representation by Proposition \ref{pro:algLCP} (3), its image $\nabla^\theta(\mg)$ is a solvable Lie subalgebra of $\so(\mmu)\oplus\R$. As $\so(\mmu)\oplus\R$ is of compact type, $\nabla^\theta(\mg)$ is of compact type as well and solvable, thus abelian.  Consequently, $0=[\nabla^\theta_x,\nabla^\theta_y](\mmu)=\nabla^\theta_{[x,y]}(\mmu)$ for all $x,y\in\mg$, hence $\nabla^\theta_{\mg'}(\mmu)=0$.
\end{proof}

\begin{lm}\label{lm:Ubot} Every LCP structure $(g,\theta,\mmu)$ on a unimodular solvable Lie algebra $\mg$ of dimension $n\ge 3$ is adapted, i.e. $\theta|_\mmu=0$. 
\end{lm}

\begin{proof} 
Let $\{e_i\}_{i=1}^n$ be an orthonormal basis of $\mg$ such that $\{e_i\}_{i=1}^s$ spans the derived algebra $\mg'$. If $s=0$ we have $\mmu=0$ by Corollary \ref{cor:abelian}, so the statement is clear. Assume now $s\ge 1$. 

For every $y\in \mmu$, we have\begin{equation}\label{eq:trsn}
0=H^\mg(y)=\tr(\ad_y)=\sum_{i=1}^n\lela \ad_y e_i,e_i\rira=\sum_{i=1}^s\lela \ad_y e_i,e_i\rira.
\end{equation}

On the other hand, taking $x=z=e_i$ and $y\in\mmu$ in \eqref{eq:LCP}, summing over $i=1, \ldots, s$, and using \eqref{eq:tg} together with Lemma \ref{lm:beta0}, we get
\begin{equation}\label{eq:sg}
0=\sum_{i=1}^s\lela [e_i,y],e_i\rira+s\theta(y) \qquad \forall\; y\in \mmu.
\end{equation}
Using \eqref{eq:trsn}, we conclude that $\theta$ vanishes on $\mmu$.
\end{proof}

We will now obtain one of the main results of the paper, which shows in particular that every LCP structure on an unimodular solvable Lie algebra is obtained by the construction of Proposition \ref{pro:cnstr}.

\begin{teo} \label{teo:UUbot}
Let  $(g,\theta,\mmu)$ be an LCP structure on a unimodular solvable Lie algebra $\mg$. Then the following holds:
\begin{enumerate}
\item $\mmu$ is an abelian ideal of $\mg$ contained in the centre $\mz(\mg')$ of $\mg'$;
\item $\nabla^\theta_xy=\ad_xy$ for every $x\in \mg$ and $y\in\mmu$.
\end{enumerate}
\end{teo}

\begin{proof}If $\mmu=0$, the result is trivially satisfied, so we assume $\mmu\neq 0$.

Consider the unique orthogonal decomposition of $\mg$
\begin{equation}
\mg=\mv_0\oplus\mv_1\oplus \mv_2,
\end{equation}satisfying $\mv_2=[\mg',\mg']\subset\mg'$, $\mg'=\mv_1\oplus \mv_2$ and 
$\mv_0=(\mg')^\bot$.

We claim that $ \mmu\bot\mv_2$. Indeed, taking $y\in\mmu$ in \eqref{eq:LCP} and using Lemma \ref{lm:Ubot} yields
 \begin{multline}
\label{eq:LCPu}
\lela\nabla^\theta_xy,z\rira=\frac12\left(\lela  [x,y],z\rira-\lela[x,z],y\rira-\lela[y,z],x)\right)\\
+\theta(x)\lela y,z\rira-\theta(z)g(x,y)\qquad\qquad \forall\; x,z\in \mg,\ \forall\; y\in\mmu.
\end{multline}
Then, taking $x,z\in\mg'$ in \eqref{eq:LCPu}, and using Lemma \ref{lm:beta0} together with \eqref{eq:tg} we obtain
\begin{equation}\label{eq:LCPu1}
    \lela [x,y],z\rira-\lela [x,z],y\rira+\lela[z,y],x\rira=0 \quad \forall\; x,z\in\mg', \ \forall\; y\in \mmu.
\end{equation}
Interchanging $x,z$ and subtracting the two equations, we obtain
\[\lela [x,z],y\rira=0 \qquad \forall\; x,z\in\mg', \ \forall\; y\in \mmu.
\]This is equivalent to $\mv_2=[\mg',\mg']\bot \mmu$, and thus $ \mmu\subset \mv_0\oplus\mv_1$.

We shall show that actually $ \mmu$ is contained in $\mv_1$. To do this, we evaluate \eqref{eq:LCPu} on $x, z\in\mv_0=(\mg')^\bot$, $y\in\mmu$, which gives
\begin{equation}\label{eq:LCPv0}
    \lela\nabla^\theta_xy,z\rira=-\frac12\lela[x,z],y\rira
+\theta(x)\lela y,z\rira-\theta(z)g(x,y)\qquad \forall\; x,z\in \mv_0,\ \forall\; y\in\mmu.
\end{equation}
In particular, taking $z=\theta^\sharp$ in the equation above and using that $\theta|_\mmu=0$ (by Lemma \ref{lm:Ubot}), we get
\begin{equation}\label{eq:v0}
 0=-\frac12\lela[x,\theta^\sharp],y\rira-|\theta|^2g(x,y)\qquad \forall\; x\in \mv_0,\ \forall\; y\in\mmu.
\end{equation}
Let us decompose some arbitrary vector $x\in\mv_0$ as $x=x_0+x_1$, where $x_0\in \mmu$ and $x_1\in \mmu^\bot$. Equation \eqref{eq:v0} applied to $y=x_0$ gives
\begin{equation}\label{eq:v01}
2|x_0|^2|\theta|^2=-\lela[x,\theta^\sharp],x_0\rira.
\end{equation}
However, by Proposition \ref{pro:algLCP} we have 
$$\lela[x,\theta^\sharp],x_0\rira=\lela[x_0,\theta^\sharp],x_0\rira+\lela[x_1,\theta^\sharp],x_0\rira=-|x_0|^2|\theta|^2.$$
since $[x_1,\theta^\sharp]\in\mmu^\perp$. This shows that $x_0=0$ and thus $x\in \mmu^\bot$. Therefore $\mv_0\subset \mmu^\bot$ which together with $ \mmu\bot\mv_2$ implies $ \mmu\subset \mv_1$. 
We shall denote by $\mm:=\mv_1\cap  \mmu^\bot$ so that $\mv_1= \mmu\oplus \mm$ as an orthogonal direct sum. The above decomposition refines to 
\begin{equation}
\mg=\mv_0\oplus \mmu\oplus\mm\oplus \mv_2,
\end{equation}

Using the fact that $\nabla^\theta$ preserves $\mmu$ together with
$\mmu\perp \mv_0$ in \eqref{eq:LCPv0} yields $[\mv_0,\mv_0]\bot \mmu$, so 
\begin{equation}\label{eq:vvmv}
[\mv_0,\mv_0]\subseteq \mm\oplus\mv_2.
\end{equation}

Consider now $x\in \mv_1\subset \mg'$ and $z\in\mg'$ in \eqref{eq:LCPu1}. The obvious inclusions $[\mv_1,\mg']\subset[\mg',\mg']\bot \mmu$ and $[\mg',\mv_1]\subset \mv_2\bot \mv_1$ show that the last two terms on the left hand side vanish. We are left with:
\begin{equation*}
\lela [x,y],z\rira=0\qquad\forall\; x\in\mv_1,\ \forall\; y\in \mmu,\ \forall\; z\in\mg'.
\end{equation*} Since $ [x,y]\in\mg'$, we conclude that 
\begin{equation}\label{eq:v1u}
[\mv_1, \mmu]=0.
\end{equation}

Notice that, since $\mg $ is solvable,  $\mg'=\mv_1\oplus\mv_2$ is a nilpotent Lie algebra with derived subalgebra $(\mg')'=\mv_2$, and $ \mmu\subset \mg'$. Let $\mk$ denote the Lie subalgebra of $\mg'$ generated by $\mv_1$. It is clear that $[\mmu,\mk]=0$ by \eqref{eq:v1u} and the Jacobi identity. Moreover, since $\mk$ is a subalgebra of $\mg'$ and clearly satisfies $\mk+[\mg',\mg']=\mg'$ we obtain $\mg'=\mk$ by \cite[Lemma 4.3]{ArLa}. Therefore, 
\begin{equation}\label{eq:g0gp}
[ \mmu,\mg']=0.
\end{equation}

Actually, $ \mmu$ is an ideal of $\mg$; to show this, it is sufficient by \eqref{eq:g0gp} to show that $[\mv_0, \mmu]\subset  \mmu$. We evaluate \eqref{eq:LCPu} on $x\in\mv_0$ and $z\in\mg'$, and use  \eqref{eq:tg} together with \eqref{eq:g0gp}, to get
\begin{equation}\label{eq:p1}
    \lela\nabla^\theta_xy,z\rira=\frac12\left(\lela  [x,y],z\rira-\lela[x,z],y\rira\right)
+\theta(x)\lela y,z\rira\qquad \forall\; x\in\mv_0,\ \forall\; y\in \mmu,\ \forall\; z\in\mg'.
\end{equation}
Next, we evaluate \eqref{eq:LCPu} on $x\in\mg'$ and $z\in\mv_0$. Using \eqref{eq:g0gp} and Lemma \ref{lm:beta0}, we obtain 
\begin{equation}
0=\frac12\left(-\lela [x,z],y\rira+\lela[z,y],x\rira\right)-\theta(z)\lela x,y\rira\qquad\forall\; x\in\mg',\ \forall\; y\in \mmu,\ \forall\; z\in\mv_0.
\end{equation}
Interchanging $x$ with $z$ in this last equation and adding the result to \eqref{eq:p1} yields
\begin{equation}
    \label{eq:p2}
 \lela\nabla^\theta_xy,z\rira=\lela [x,y],z\rira\qquad\forall\; x\in\mv_0,\ \forall\; y\in \mmu,\ \forall\; z\in\mg'.
\end{equation}
If $z\in\mv_0$, both terms of \eqref{eq:p2} vanish, so this equation actually holds for every $z\in\mg$. We thus obtain for all $x\in\mv_0$ and $y\in \mmu$:
\begin{equation}\label{eq:v0g0rule}
[x,y]=\nabla^\theta_xy\in \mmu.
\end{equation}
Note that this equation actually holds for every $x\in \mg$, since both terms vanish for $x\in \mg'$ by \eqref{eq:g0gp} and Lemma \ref{lm:beta0}. This proves (2). Moreover, $\mmu$ is an ideal of $\mg$, contained in the centre of $\mg'$ by \eqref{eq:g0gp}, thus proving (1).
\end{proof}

We now derive some consequences of Theorem \ref{teo:UUbot}. We start by showing that there are no interesting LCP structures on nilpotent Lie algebras.

\begin{cor}
Every LCP structure on a nilpotent Lie algebra is degenerate.
\end{cor}
\begin{proof} 
Assume that $(\mg,g,\theta,\mmu)$ is a nilpotent LCP Lie algebra, and let $q$ denote the dimension of $\mmu$. Then, $\mg$ is unimodular and $\mmu^\bot$ is nilpotent, being a Lie subalgebra of a nilpotent Lie algebra, hence it is unimodular. By the second part of Equation \eqref{eq:trace} in Corollary \ref{cor:3.4} 
we get $q\theta|_{\mmu^\perp}=0$. However, $\theta|_{\mmu^\perp}\ne 0$ since $\theta\ne 0$ and $\theta|_\mmu=0$ by Lemma \ref{lm:Ubot}, so finally $q=0$, i.e. the LCP structure is degenerate.
\end{proof}

Theorem \ref{teo:UUbot} shows in particular that any unimodular solvable LCP Lie algebra $(\mg,g,\theta,\mmu)$ is a semidirect product of the non-unimodular Lie algebra $\mh:=\mmu^\bot$, acting on the abelian Lie algebra $\mmu$ via the Lie algebra representation $\nabla^\theta|_{\mh}:\mh\to \gl(\mmu)$, as in Proposition \ref{pro:cnstr}. As a consequence of these two results, we have the following:

\begin{cor}\label{cor:11}
There is a one-to-one correspondence between non-degenerate unimodular solvable LCP Lie algebras $(\mg,g,\theta,\mmu)$  and triples $(\mh,h,\beta)$, where $(\mh,h)$ is a metric solvable non-unimodular Lie algebra, and $\beta:\mh\to \so(q)$ is a Lie algebra representation, for some $q\in\N^*$, which vanishes on $\mh'$.
\end{cor}

In particular, Corollary \ref{cor:11} shows that every modification of the metric on the orthogonal $\mh=\mmu^\perp$ of the flat space of an LCP structure on a solvable unimodular Lie algebra does not alter the LCP character:

\begin{cor}\label{cor:metric}
If $\mg$ is solvable and unimodular, and $(\mg,g,\theta,\mmu)$ is an LCP Lie algebra, then $(\mg,\tilde g,\theta,\mmu)$ is also an LCP Lie algebra for any other scalar product $\tilde g$ with the property that $\mmu\subset \ker (\tilde g-g)$.
\end{cor}

Consequently, in the solvable unimodular case, there are more general modifications of the metric than the one in Corollary \ref{cor:modif} that preserve the LCP condition.

As another application of Theorem \ref{teo:UUbot} we obtain the description of LCP structures on unimodular almost abelian Lie algebras.

\begin{cor}\label{cor:almost-abelian} Let $\mg$ be an unimodular almost abelian Lie algebra. If $\mg$ admits an LCP structure $(g,\theta,\mmu)$, then $\mg$ can be written as a semidirect sum $\mg=\R b\ltimes (\R^p\oplus \R^q)$, where the three factors are orthogonal with respect to $g$, $g(b,b)=1$, and $\mmu=\R^q$. Moreover, $\R^p$ and $\R^q$ are invariant by $\ad_b$, which has the form
\begin{equation}
\ad_b|_{\R^p\oplus\R^q}= \left[\begin{array}{c|c}  
 A &  \\ \hline 
  & B-\frac1q\tr(A)\Id_{\R^q}
\end{array}\right],
\end{equation}
for some $A\in \gl(p)$ such that $\tr(A)\neq 0$, $B\in \so(q)$, and $\theta=-\frac1q\tr(A)g(b,\cdot)$. In particular, the LCP Lie algebra is isomorphic to $\mathcal{L}_{(A,B)}$ constructed in Example $\ref{ex:almab}$.
\end{cor}

\begin{proof}
Every almost abelian metric Lie algebra $(\mg, g)$ can be written as a semidirect product $\mg:=\R b \ltimes_C \R^k$, where $\R^k$ is a codimension 1 abelian ideal, $b$ is a unit length vector orthogonal to $\R^k$, and $C:=\ad_b\in\mathfrak{gl}(k)$. The unimodularity condition is equivalent to $\tr (C)=0$.

If $(g,\theta,\mmu)$ is an LCP structure on $\mg$, then by Theorem \ref{teo:UUbot}, $\mmu\subset \mg'\subset \R^k$, so identifying $\mmu$ with $\R^q$ and its orthogonal in $\R^k$ with $\R^p$, we have $\mmu^\bot= \R b\oplus \R^p$. Note that since $\mmu$ is an ideal of $\mg$ and $\mmu^\bot$ is a subalgebra, $C$ preserves $\R^p$ and $\R^q$, so with respect to the orthogonal direct sum decomposition $\R^k=\R^p\oplus\R^q$, $C$ can be written in a block-diagonal form as 
$$C=\left[\begin{array}{c|c}  
 A &  \\ \hline 
  & D
\end{array}\right],$$
with $A\in\mathfrak{gl}(p)$ and $D\in\mathfrak{gl}(q)$. The trace form of the Lie algebra $\mmu^\bot=\R b\ltimes_B\R^p$ is $H^{\mmu^\bot}=\tr(A)g(b,\cdot)$.
By Lemma \ref{lm:beta0} and Corollary \ref{cor:3.4} we have $\theta=\theta|_{\mmu^\perp}=-\frac1qH^{\mmu^\bot}=-\frac1q \tr (A) g(b,\cdot)$ and $\ad_b|_\mmu-\theta(b)\Id_{\R^q}\in \so(q)$. Thus $D=\ad_b|_\mmu$ can be written as $B-\frac1q \tr(A)\Id_{\R^q}$ for some $B\in \so(q)$, showing that $(\mg,g,\theta,\mmu)$ is isomorphic to the LCP Lie algebra $\mathcal{L}_{(A,B)}$ constructed in Example \ref{ex:almab}.
\end{proof}

We finish this section by studying the amalgamated product of two unimodular almost abelian LCP Lie algebras. 

\begin{ex}\label{ex:amal1}
Let $(\mg_i, g_i, \theta_i, \mmu_i)$, for $i=1,2$, denote a unimodular almost abelian LCP Lie algebra. Then we may write $\mg_i=\R b_i\ltimes_{C_i} \R^{n_i}$, where $b_i\perp \R^{n_i}$,  $|b_i|=1$, and $C_i=\ad_{b_i}|_{\R^{n_i}}\in\gl(n_i)$. It follows from Corollary \ref{cor:almost-abelian} that $\theta_i^\sharp$ is a multiple of $b_i$; moreover, replacing $b_i$ by $-b_i$ if necessary, we may assume that  $\theta_i=|\theta_i|g_i(b_i,\cdot)$ for $i=1,2$. Since any LCP structure on a unimodular solvable Lie algebra is adapted (Lemma \ref{lm:Ubot}), we can define the amalgamated product of $(\mg_1,g_1,\theta_1,\mmu_1)$ and $(\mg_2,g_2,\theta_2,\mmu_2)$, with underlying Lie algebra denoted by $\mg$. Recall that $\mg=\ker(\theta_1-\theta_2)\subset \mg_1\oplus \mg_2$, where $\theta_1$ and $\theta_2$ denote the obvious extensions to $\mg_1\oplus \mg_2$ of $\theta_1\in\mg_1^*$ and $\theta_2\in\mg_2^*$.  Note that $|\theta_2|b_1+|\theta_1|b_2\in\mg$, $\R^{n_1}\oplus \R^{n_2}\subset \mg$ and, moreover, $|\theta_2|b_1+|\theta_1 |b_2 \perp \R^{n_1}\oplus \R^{n_2}$ with respect to $g=(g_1+g_2)|_\mg$. Furthermore, if we denote by $b$ the unit vector in the same direction as $|\theta_2|b_1+|\theta_1|b_2$, i.e.,  \begin{equation}\label{eq:b} 
b:=\frac{1}{(|\theta_1|^2+|\theta_2|^2)^\frac12}(|\theta_2|b_1+|\theta_1|b_2),\end{equation}
we have that $\mg$ is a unimodular almost abelian Lie algebra itself, since we may write 
\[ \mg= \R b  \ltimes_C (\R^{n_1}\oplus \R^{n_2}) ,\] where $C$ is the following matrix:
\begin{equation}\label{eq:C}
C=\frac{1}{(|\theta_1|^2+|\theta_2|^2)^\frac12} \left[ \begin{array}{c|c} |\theta_2| C_1 & \\ \hline
&  |\theta_1| C_2
\end{array} \right].
\end{equation}
 Clearly, $\tr (C)=0$.
 
The closed $1$-form $\theta$ that defines the LCP structure is a multiple of $g(b,\cdot)$. More explicitly, it follows from Corollary \ref{cor:almost-abelian} that $\mmu_i^\perp =\R b_i \oplus \R^{p_i}$ (orthogonal sum) and $|\theta_i|=-\frac{1}{q_i} \tr(A_i)$ where $q_i=\dim \mmu_i$ and $A_i=\ad_{b_i}|_{\R^{p_i}}:\R^{p_i}\to \R^{p_i}$. Note that $\tr(A_i)<0$. Since $\mmu$, the flat subspace of the LCP structure on $\mg$, is given by $\mmu=\mmu_1\oplus \mmu_2$, we have that $\theta=|\theta|g(b,\cdot)$, with
\[ |\theta| =-\frac{1}{q_1+q_2}(\tr(A_1) +\tr(A_2)).\]
\end{ex}

\medskip

\section{LCP structures with flat subspace of small codimension}

We start this section by showing that the flat subspace of a non-degenerate LCP structure on a unimodular solvable Lie algebra has codimension at least two.

\begin{pro}\label{pro:ubotgeq2}
Any non-degenerate LCP Lie algebra $(\mg,g,\theta,\mmu)$ on a unimodular solvable Lie algebra $\mg$ of dimension $n$ verifies $\dim \mmu\leq n-2$. 
\end{pro}
\begin{proof}Let $(\mg,g,\theta,\mmu)$ be a unimodular solvable  
LCP Lie algebra. By Lemma \ref{lm:Ubot}, the (non-zero) 1-form $\theta$ vanishes on $\mmu$, so $\dim \mmu\leq n-1$. 
Moreover, by Theorem \ref{teo:UUbot} (2), together with \eqref{eq:weylc2}, one can write:
\[
(\ad_{\theta^\sharp})|_\mmu=|\theta|^2\Id_\mmu+\beta(\theta^\sharp),
\]with $\beta(\theta^\sharp)\in\so(\mmu)$. 
If  $\dim \mmu=n-1$, then $\mg=\R \theta\oplus \mmu$, so as $\ad_{\theta^\sharp}(\theta^\sharp)=0$, the previous equation yields  $$0=\tr\ad_{\theta^\sharp}=(n-1)|\theta|^2,$$ contradicting the fact that $\theta\ne 0$. Thus $\dim \mmu\leq n-2$.
\end{proof}

We will now show that every unimodular solvable Lie algebra $\mg$ carrying an LCP structure $(g,\theta,\mmu)$ such that the codimension of the flat subspace is 2, is almost abelian.

\begin{pro}\label{pro:ubot2} Any solvable unimodular Lie algebra $\mg$ of dimension $n\ge 3$ admitting  an LCP structure $(g,\theta,\mmu)$ with $\dim \mmu=n-2$ is almost abelian. 
\end{pro}

\begin{proof}
Since $\mg$ is solvable and  unimodular, Theorem \ref{teo:UUbot} shows that $\mg$ has a semidirect product decomposition as
$\mg=\mmu^\bot \ltimes \mmu$ such that $\mmu\subset\mg'$ and $[\mmu,\mg']=0$. In addition, by \eqref{eq:weylc2} we have
\begin{equation}\label{eq:adad}
\ad_z|_\mmu-\theta(z)\Id_\mmu=\nabla_z^\theta|_\mmu-\theta(z)\Id_\mmu\in\so(\mmu)\qquad \forall\; z\in \mmu^\bot. 
\end{equation}

Let $\{b,x\}$ be an orthonormal basis of $\mmu^\bot$ such that $\theta(b)=:\mu>0$, $\theta(x)=0$. Since $\mmu^\bot$ is a subalgebra,  $[b,x]$ is a linear combination of $b$ and $x$. However, using that $[\mmu,\mg']=0$ we get
$$
\ad_{[b,x]}|_\mmu=0,
$$
so, in particular, the symmetric part of $\ad_{[b,x]}|_\mmu$ (which is $\theta([b,x])$ by \eqref{eq:adad} and the choice of $b$), vanishes. Hence, $[b,x]=\lambda x$ for some $\lambda\in\R$. In fact, since $\mg$ is unimodular and $\tr\ad_b|_\mmu=(n-2) \mu$ (see \eqref{eq:adad}), we have $[b,x]=- (n-2)\mu\,x$, which is non-zero. 

In particular, the above shows that $x\in\mg'$ and thus $x$ commutes with $\mmu$. Therefore, $\mk:=\R x\oplus \mmu$ is an abelian ideal of $\mg$, so that  $\mg$ is an almost abelian Lie algebra.
\end{proof}

The next result treats the case where the flat space $\mmu$ has codimension 3. 

\begin{pro}\label{pro:ubot3}
Let $\mg$ be a unimodular solvable  Lie algebra of dimension $n$ admitting a non-degenerate LCP structure $(g,\theta,\mmu)$ with $\dim \mmu=n-3$. Then  either $\mg$ is almost abelian, or else $n\geq 5$ and there is an orthonormal basis $\{b,x,y\}$ of $\mmu^\bot$ such that $\theta(b)=\mu>0$,  $\theta(x)=\theta(y)=0$ and the  matrices of the adjoint maps in the basis $\{b,x,y\}$ verify
\[ \ad_b = \left[\begin{array}{ccc|c}  
 0 &0&0& \\
 0&(3-n)\mu & r &  \\ 
 0&0 & 0 &  \\ \hline 
 &&  & \mu \Id_\mmu+B_1 
\end{array}\right], \quad 
\ad_y = \left[\begin{array}{ccc|c}  
  0 & 0 & 0 & \\
 -r & 0 & 0 &  \\ 
 0 & 0 & 0 &   \\ \hline 
 & &  & B_2
\end{array}\right], \]
where $r\in\R$, $B_1,B_2\in\so(\mmu)$, $B_2\neq 0$ and $[B_1,B_2]=0$. In addition, the LCP Lie algebra is isomorphic to $\mathcal C_{A,B_1,B_2,v}$ in Example $\ref{ex:noalmab}$, where $A:=(n-3)\mu\in\R=\mathfrak{gl}(1)$ and $v:=rx$. 
\end{pro}

\begin{proof}
We proceed as in the previous proposition. By Theorem \ref{teo:UUbot}, $\mg=\mmu^\bot\ltimes \mmu$, where $\mmu$ is an abelian ideal and $\mmu^\bot$ is a non-unimodular solvable Lie subalgebra of dimension 3 such that $\theta|_\mmu=0$ (see Lemma \ref{lm:Ubot}). 

Define $W:=\ker \theta\cap \mmu^\bot$, and notice that $0\neq [\mmu^\bot,\mmu^\bot]\subset W$, since $\theta$ is closed. Let $b$ be a unit element in $W^\bot$ such that $\theta(b)=:\mu>0$. Since $[\mmu^\bot,\mmu^\bot]\subset W$, $\ad_b$ preserves $W$ and
\begin{equation}\label{eq:adth}
\ad_b=(\ad_b)|_W+\mu\Id_\mmu+\beta(b),
\end{equation}where $\beta:\mg\to\so(\mmu)$ is a Lie algebra representation. Unimodularity implies 
\begin{equation}
    \label{eq:traza}
-\tr (\ad_b)|_W=\tr (\ad_b)|_\mmu=\mu(n-3).
\end{equation}

Fix an orthonormal basis $\{x,y\}$  of $W$ such that $x\in [\mmu^\bot,\mmu^\bot].$ We claim that $W$ is an abelian subalgebra. 

If $W=[\mmu^\bot,\mmu^\bot]$, it is straightforward that it is abelian. Indeed, the commutator of a solvable Lie algebra $\mmu^\bot$ is nilpotent, but being of dimension $\leq 2$, it must be abelian.

If $W\neq[\mmu^\bot,\mmu^\bot]$, then $[b,y]=r x$ and $[y,x]=s x$ for some $r,s\in\R$ so that $\tr (\ad_y)|_{\mmu^\bot}=s$. However, since $ \theta(y)=0$, we have 
\begin{equation}\label{eq:betay}
(\ad_y)|_\mmu=\beta(y)\in \so(\mmu),
\end{equation} and thus $0=\tr (\ad_y)|_\mmu=-\tr(\ad_y)|_{\mmu^\bot}=s$. Therefore, $x$ and $y$ commute and $W$ is abelian in this case as well, proving our claim.

Recall that $[\mg',\mmu]=0$. Therefore,  if either $W=[\mmu^\bot,\mmu^\bot]$, or $W\neq [\mmu^\bot,\mmu^\bot]$ and $\beta(y)=0$ in \eqref{eq:betay}, then $\mk:=W\oplus \mmu$ is an abelian ideal of $\mg$. Hence $\mg$ is almost abelian in these cases.

Assume now that $W\neq [\mmu^\bot,\mmu^\bot]$ and $\beta(y)\neq 0$ in \eqref{eq:betay}; notice that the latter implies $\dim\mmu\geq 2$ and thus $n\geq 5$.  Since the commutator of $\mmu^\bot$ is 1-dimensional, we get from \eqref{eq:traza}:
\[
\ad_b=-(n-3)\mu x\otimes x+ry\otimes x+\mu \Id_\mmu+\beta(b).
\]
It is possible to check that this LCP structure arises from Example \ref{ex:noalmab} by taking
$q=n-3$, $p=1$ and $\R^p=\R x$, $A=(n-3)\mu$, $v=rx$, $B_1=\beta(b)$ and $B_2=\beta(y)$. 
\end{proof}

\medskip

\section{Low dimensional LCP Lie algebras}\label{sec:classification}

In this section we study low dimensional solvmanifolds equipped with LCP structures. We begin by classifying the unimodular solvable Lie algebras of dimension at most 5 that admit LCP structures. In the second part of the section we deal with the existence of lattices in the associated Lie groups. Since the existence of lattices does not depend on the left invariant Riemannian metric on the group, we will be interested only in the isomorphism classes of LCP Lie algebras, rather than isometry classes.

The summary of the classification will be given in Tables \ref{table:dim3}-\ref{table:dim5-1} in the Appendix, where the name of each Lie algebra and the (non-)existence of lattices in the associated simply connected Lie group is provided.

\begin{pro}\label{pro:dim3} Let $\mg$ be a  $3$-dimensional unimodular solvable Lie algebra. Then $\mg$ admits a non-degenerate LCP structure if and only if it is isomorphic to the Lie algebra $\mathfrak{e}(1,1)$ in Table $\ref{table:dim3}$.
\end{pro}

\begin{proof}
Every  non-degenerate LCP Lie algebra $(\mg,g,\theta,\mmu)$ with $\dim \mg=3$, verifies $\dim \mmu^\bot=2$ and $\dim \mmu=1$ due to Proposition \ref{pro:ubotgeq2}.

From Corollary \ref{cor:almost-abelian} it follows that $\mg$ is almost abelian and by Proposition \ref{pro:ubot2}, it can be written as an orthogonal semidirect product $\R b\ltimes(\R x\oplus \R u)$, where $b,x,u$ are of unit length and 
\[
\ad_{b}=\lambda x\otimes x-\lambda u\otimes u,\qquad [x,u]=0,
\]for some $\lambda\neq 0$. Taking $e_1=\frac{1}{\lambda}b$, $e_2=x$ and $e_3=u$, we see that $\mg$ is isomorphic to $\mathfrak{e}(1,1)$.

Conversely, $\mathfrak{e}(1,1)$ is isomorphic to the Lie algebra of $\mathcal L_{(\lambda \Id_\R,0)}$ in Example \ref{ex:almab}, so it admits an LCP structure.
\end{proof}

The LCP structure obtained on $\mathfrak{e}(1,1)$ coincides with the one  constructed by Matveev and Nikolayevsky in \cite{MN2015}.

We consider now the $4$-dimensional case. In the next result we will use the notation of Bock in \cite{Bo16} which, in turn, is borrowed from \cite{Mu}.

\begin{pro} \label{pro:dim4}
Let $\mg$ be a $4$-dimensional unimodular solvable Lie algebra. Then $\mg$ admits a non-degenerate LCP structure if and only if $\mg$ is isomorphic to one of the Lie algebras in Table $\ref{table:dim4}$. 
\end{pro}

\begin{proof}
If $\mg$ is of dimension 4 and admits an LCP structure $(g,\theta,\mmu)$, then $\dim \mmu$ is either 1 or 2, in view of Proposition \ref{pro:ubotgeq2}. In addition, by Propositions \ref{pro:ubot2} and \ref{pro:ubot3}, we have that $\mg$ is almost abelian. Therefore, by Corollary \ref{cor:almost-abelian}, $\mg$ can be written as an orthogonal semidirect sum $\mg=\R b\ltimes (\R^p\oplus\R^q)$, where $p+q=3$ and 
\begin{equation}
\ad_b|_{\R^p\oplus\R^q}= \left[\begin{array}{c|c}  
 A &  \\ \hline 
  & B-\frac1q\tr(A)\Id_{\R^q}
\end{array}\right],
\end{equation}
for some $A\in \gl(p)$ with $\tr(A)\neq 0$ and $B\in \so(q)$.

Assume first that $p=1$, $q=2$ and consider an orthonormal basis  $\{b,x,u,v\}$ of $\mg$ such that $\R^p=\R x$, $\R^q={\rm span}\{u,v\}=\mmu$, and  the action of $b$ on $\R^p \oplus \R^q$ is given by 
\begin{equation}\label{eq:3x3i}
\ad_b|_{\R^p \oplus\R^q}= \left[\begin{array}{c|cc} 
\lambda &&\\\hline
&-\frac{\lambda}{2} & -a \\
& a & -\frac{\lambda}{2}   \end{array}\right], 
\end{equation}  
for some $\lambda,a\in\R$, $\lambda\neq 0$.

If $a=0$, taking $e_1=\frac{1}{\lambda}b$, $e_2=x,\, e_3=u$ and $e_4=v$ we obtain the Lie algebra $\mg_{4.5}^{-\frac12,-\frac12}$. When $a\neq 0$, taking $e_1=-\frac{1}{a}b$, $e_2=x,\, e_3=u$ and $e_4=v$ we obtain the Lie algebra $\mg_{4.6}^{-2p,p}$ from Table 2 for $p=\frac{\lambda}{2a}\neq 0$.

Suppose now that $p=2$, $q=1$ and let  $\{b,x,y,u\}$ be an orthonormal basis such that  $\R^p={\rm span}\{x,y\}$ and $\R^q=\R u$, 
\[
\ad_b|_{\R^p\oplus \R^q}= \left[\begin{array}{c|c}  
 A&  \\ \hline 
  &-\tr(A) 
\end{array}\right], 
\] 
for some $A\in \gl(2)$ such that $\tr(A)\neq 0$. 

We consider now the possible Jordan forms of the matrix $A$ to identify the structure of the Lie algebra $\mg$.

Assume first that $A$ has real eigenvalues and it is diagonalizable, so that in a certain basis it has the form
\[
A=\begin{bmatrix} \alpha & 0\\0 & \beta\end{bmatrix}
    \]
    for some $\alpha,\beta\in \R$ such that $\alpha+\beta\neq 0$. If $\alpha=0$, so that $\beta\neq 0$, replacing $b$ with $\frac{b}{\beta}$ we obtain the Lie algebra $\mathfrak{e}(1,1)\oplus \R$. The same happens if $\beta=0$ and $\alpha\neq 0$. 
    However, if $\alpha\neq 0$ and $\beta\neq 0$, with $\alpha+\beta\neq 0$, setting $e_1=-\frac{1}{\alpha+\beta}b$, $e_2=u$, $e_3=x$ and $e_4=y$, we obtain that $\mg$ is isomorphic to $\mg_{4.5}^{p,-p-1}$ in Table \ref{table:dim4} for $p=-\frac{\alpha}{\alpha+\beta}$.
    
When $A$ has real eigenvalues but it is not diagonalizable then we may assume 
\[A=\begin{bmatrix} \alpha & 1\\ 0& \alpha \end{bmatrix}
    \]
    for some $0\neq \alpha\in \R$. Setting $e_1=\frac{b}{\alpha}$, $e_2=\frac{x}{\alpha}$, $e_3=y$ and $e_4=u$ we obtain the Lie algebra $\mg_{4.2}^{-2}$ in Table \ref{table:dim4}.

Finally, if $A$ has non-real eigenvalues then we may assume \[A=\begin{bmatrix} \alpha & \beta\\ -\beta & \alpha\end{bmatrix}\] for some $\alpha,\beta\in \R$ with $\alpha\neq 0$ and $\beta\neq 0$. Setting $e_1=\frac{b}{\beta}$, $e_2=u$, $e_3=x$ and $e_4=y$ we obtain the Lie algebra $\mg_{4.6}^{-2p,p}$ in Table \ref{table:dim4} for $p=\frac{\alpha}{\beta}\neq 0$.

Conversely, an explicit LCP structure on $\mathfrak{e}(1,1)\oplus \R$ can be given by extending the LCP structure on $\mathfrak{e}(1,1)$ using Proposition \ref{pro:direct}. The remaining Lie algebras listed above are isomorphic to $\mathcal L_{(A,B)}$ in Example \ref{ex:almab}, for certain $A\in \gl(p)$ with $\tr(A)\neq 0$ and $B\in\so(q)$; each one of them admits non-degenerate LCP structures.
\end{proof}

The summary of the $4$-dimensional case is given in Table \ref{table:dim4} in the Appendix.

Notice that the Lie algebras $\mg_{4.5}^{-\frac12,-\frac12}$ and $\mg_{4.6}^{-2p,p}$ admit LCP structures with dimension of the flat factor either 1 or 2, whilst all the other Lie algebras in the table admit only one possibility for the dimension of $\mmu$.

\

We consider next the $5$-dimensional case. We will divide it in $3$ different cases, depending on the dimension of the flat subspace $\mmu$, which according to Proposition \ref{pro:ubotgeq2} can be equal to $1,\,2$ or $3$.

\begin{pro}\label{pro:dim5-3}
Let $\mg$ be a $5$-dimensional unimodular solvable Lie algebra. Then $\mg$ admits a non-degenerate LCP structure with $3$-dimensional flat subspace if and only if it is isomorphic to either $\mg_{5.13}^{-\frac13,-\frac13,r}$ or $\mg_{5.7}^{\frac13,\frac13,\frac13}$ in Table $\ref{table:dim5-1}$.
\end{pro}

\begin{proof}
According to Proposition \ref{pro:ubot2}, $\mg$ is almost abelian and thus, by Corollary \ref{cor:almost-abelian} isomorphic to the underlying Lie algebra of  $\mathcal L_{(A,B)}$. 

Since $\dim \mg=5$ and $\dim \mmu=3$, this implies that the Lie algebra can be written as $\mg=\R b\ltimes (\R x\oplus \R^3)$ with $\mmu=\R^3$, such that $\R x\oplus \R^3$ is an abelian ideal. Moreover, the action of $b$ on this ideal is given by\[
\ad_b|_{\R x\oplus \R^3}= \left[\begin{array}{c|c}  
 -3\mu &  \\ \hline 
  & \mu \Id_\mmu+B
\end{array}\right], 
\] 
for some $0\neq \mu\in\R$ and $B\in\so(3)$. Setting $b':=\frac{1}{\mu}b$ and using the normal form of any matrix in $\so(3)$, we may change the basis of $\mmu$ so that 
\[ \ad_{b'}|_\mk= \left[\begin{array}{c|ccc}  
 -3 & & & \\ \hline 
  & 1 & 0 & 0 \\
  & 0 & 1 & -p \\
  & 0 & p & 1
\end{array}\right], 
\] 
for some $p\in\R$. 

It is straightforward to check that  $\mg$ is isomorphic to $\mg_{5.13}^{-\frac13,-\frac13,r}$ for $r=-\frac{p}3\neq 0$, and isomorphic to $\mg_{5.7}^{\frac13,\frac13,\frac13}$ otherwise (see Table \ref{table:dim5-1} in the Appendix). 

For the converse, it is enough to notice that, by construction, the  Lie algebras $\mg_{5.13}^{-\frac13,-\frac13,r}$ and $\mg_{5.7}^{\frac13,\frac13,\frac13}$ are of the type $\mathcal L_{(A,B)}$ for certain $0\neq A\in \R=\gl(1)$ and $B\in\so(3)$. Hence they admit an LCP structure with 3-dimensional flat space as in Example \ref{ex:almab}.
\end{proof}

We consider now the case of a $2$-dimensional flat subspace. 

\begin{pro}\label{pro:dim5-2}
Let $\mg$ be a $5$-dimensional unimodular solvable Lie algebra. Then $\mg$  admits a non-degenerate LCP structure with $2$-dimensional flat subspace if and only if $\mg$ is  isomorphic to one of the following Lie algebras in Table $\ref{table:dim5-1}$:  
 $\mg_{4.5}^{-\frac12,-\frac12}\oplus \R$, $\mg_{4.6}^{-2p,p}\oplus\R$, $\mg_{5.7}^{1-2q,q,q}$, $\mg_{5.7}^{q,q,1-2q}$, $\mg_{5.9}^{-1,-1}$, $\mg_{5.13}^{-1-2q,q,r}$, $\mg_{5.16}^{-1,q}$, $\mg_{5.17}^{p,-p,r}$ or $\mg_{5.35}^{-2,0}$.
\end{pro}

\begin{proof}
According to Proposition \ref{pro:ubot3} for $n=5$, there are two possibilities for the structure of $\mg$:

\textbf{Case 1.} $\mg$ is almost abelian, and thus by 
Corollary \ref{cor:almost-abelian}, it has an orthonormal basis $\{b,x,y,u,v\}$ where the last elements span an abelian ideal $\mk$  of dimension 4, and the action of $b$ on $\mk$ is given by
\[ \ad_b|_\mk = \left[\begin{array}{c|cc}  
 A &  \\ \hline 
  & \mu & -a \\
  & a & \mu 
\end{array}\right], \]
for some $A\in\gl(2)$ with $\tr(A)\neq 0$ and some $\mu, a \in\R$, $\mu\neq 0$. Replacing $b$ by $-\frac{b}{\mu}$  we may assume $\mu=-1$, and exchanging $u$ and $v$ we may assume $a\ge0$. 
Since $\mg$ is unimodular we have that $\tr(A)=2$. 

We shall determine the structure of $\mg$ depending  on the Jordan form of the matrix $A$.

Assume first that $A$ has real eigenvalues and it is diagonalizable. With respect to a basis of eigenvectors, we can write
\[
A=\begin{bmatrix} \alpha & 0\\0 & \beta
    \end{bmatrix},
    \]
     with $\alpha\le\beta\in \R$ and $\alpha+\beta=2$ (thus $\alpha\leq 1$). Suppose that $a\neq 0$. If $\alpha=0$, then setting $e_1=-\frac{1}{a}b, \, e_2=y,\, e_3=u,\, e_4=v$ and $e_5=x$ we get that $\mg$ is isomorphic to $\mg_{4.6}^{-2p,p}\oplus \R$ for $p=\frac{1}{a}$. To the contrary, if $\alpha\neq 0$ then $\beta\geq 1$, thus setting $e_1=y$, $e_2=v$, $e_3=u$, $e_4=x$ and $e_5=\frac{1}{\beta}b$, we get the Lie algebra $\mg_{5.13}^{-1-2q,q,r}$ for $q=-\frac{1}{\beta}\in [-1,0]$ and $r=\frac{a}{\beta}>0$. 

    Suppose now $a=0$. If $\alpha=0$, then setting $e_1=\frac{1}{2}b,\, e_2=y,\, e_3=u,\, e_4=v$ and $e_5=x$ we get that $\mg$ is isomorphic to $\mg_{4.5}^{-\frac12,-\frac12}\oplus \R$ (see Table \ref{table:dim4}). Finally, if $\alpha\neq 0$, set $e_5=-\frac1{\alpha}b$ $e_1=y$, $e_2=u$, $e_3=v$, $e_4=x$,  $q=\frac1{\alpha}$ and $p=1-2q$.  We obtain an isomorphism with $\mg_{5.7}^{p,q,q}$ when $\alpha>0$, and with $\mg_{5.7}^{q,q,p}$ for  $\alpha<0$. 
    
Next, if $A$ has real eigenvalues but is not diagonalizable over $\R$, we may assume that \[A=\begin{bmatrix} 1 & 1\\ 0 & 1 \end{bmatrix}.\] It is clear that $\mg$ is isomorphic to $\mg_{5.9}^{-1,-1}$ when $a=0$ and to $\mg_{5.16}^{-1,q}$ for $q=a$, when $a\neq 0$.

Finally, if $A$ has non-real eigenvalues we may assume \[A=\begin{bmatrix} 1 & -\beta \\ \beta & 1
    \end{bmatrix}\] for some $\beta> 0$. If $a=0$ then setting $e_5=-b$, $e_1=u$, $e_2=x$, $e_3=y$ and $e_4=v$, we obtain that $\mg$ is isomorphic to $\mg_{5.13}^{1,-1,r}$ for $r=\beta$. Finally, if $a\neq 0$ let us set $e_5=\frac{1}{\beta}b$, $e_1=x,\, e_2=y,\, e_3=u$ and $e_4=v$, so that we get an isomorphism with the Lie algebra $\mg_{5.17}^{p,-p,r}$ for $p=\frac{1}{\beta}>0$ and $r=\frac{a}{\beta}$.

\textbf{Case 2.} there is an orthonormal basis $\{b,x,y,u,v\}$ of $\mg$ such that $\mk=\spa\{x,y,u,v\}$ is a non-abelian ideal and the Lie bracket of $\mg$ is encoded in the adjoint action of $b$ and $y$:
\[ \ad_b = \left[\begin{array}{ccc|cc}  
 0 &&&& \\
 &-2\mu & r & & \\ 
 &0 & 0 & &  \\ \hline 
 &&  & \mu & -a \\
 && & a & \mu 
\end{array}\right], \quad 
\ad_y = \left[\begin{array}{ccc|cc}  
  0 & 0 & 0 && \\
 -r & 0 & 0 & & \\ 
 0 & 0 & 0 & &  \\ \hline 
 & &  & 0 & -c \\
 & & & c & 0
\end{array}\right], \]
in the ordered basis $\{b,x,y,u,v\}$, for some $\mu,r,a,c\in\R$ with $\mu\neq 0$, $c\neq 0$.

We note first that replacing $b$ by $b-\frac{a}{c}y$ we may assume $a=0$. The other parameters in the Lie brackets remain unchanged. Setting now $b'=\frac{b}{\mu}$ and $y'=\frac{y}{c}$ we obtain:
\[ [b',x]=-2x,\; [b',y']=\frac{r}{\mu c}x, \; [b',u]=u,\; [b',v]=v,\; [y',u]=v,\; [y',v]=-u.  \]
Therefore, if we denote $e_1=\frac{r}{2\mu c}x+y',\, e_2=u,\,e_3=v,\,e_4=x, e_5=b'$ we obtain the Lie algebra $\mg_{5.35}^{-2,0}$.

To finish the proof it remains to notice that the Lie algebras $\mg_{5.7}^{1-2q,q,q}$, $\mg_{5.7}^{q,q,1-2q}$, $\mg_{5.9}^{-1,-1}$, $\mg_{5.13}^{-1-2q,q,r}$, $\mg_{5.16}^{-1,q}$, $\mg_{5.17}^{p,-p,r}$ or $\mg_{5.35}^{-2,0}$ in Table  \ref{table:dim5-1}, are either isomorphic to the Lie algebras underlying $\mathcal{L}_{(A,B)}$ (in Case 1) or $\mathcal C_{(A,B_1,B_2,v)}$ (in Case 2) described in Examples \ref{ex:almab} and \ref{ex:noalmab},  and thus they admit LCP structures.  In addition, since any LCP structure on a solvable unimodular 4-dimensional Lie algebra $\tmg$ is adapted, $\mg=\tmg\oplus\R$ carries an LCP structure by Proposition \ref{pro:direct}.
\end{proof}

\begin{pro}\label{pro:dim5-1}
Let $\mg$ be a $5$-dimensional unimodular solvable Lie algebra. Then $\mg$ admits a non-degenerate LCP structure with $1$-dimensional flat subspace $\mmu$ if and only if it is isomorphic  to one of the following Lie algebras in Table $\ref{table:dim5-1}$:
$\mathfrak{e}(1,1)\oplus \R^2$, $\mg_{4.2}^{-2}\oplus \R$, $\mg_{4.5}^{p,-p-1}\oplus \R$, $\mg_{4.6}^{-2p,p}\oplus\R$, $\mg_{5.7}^{p,q,r}$, $\mg_{5.8}^{-1}$, $\mg_{5.9}^{p,-2-p}$, $\mg_{5.11}^{-3}$, $\mg_{5.13}^{-1-2q,q,r}$, $\mg_{5.19}^{p,-2p-2}$, $\mg_{5.23}^{-4}$, $\mg_{5.25}^{p,4p}$, $\mg_{5.33}^{-1,-1}$ or $\mg_{5.35}^{-2,0}$.
\end{pro}

\begin{proof}
Assume that $(g,\theta,\mmu)$ is an LCP structure on a 5-dimensional solvable unimodular Lie algebra $\mg$ with 1-dimensional flat subspace spanned by $u\in\mg$. According to Theorem \ref{teo:UUbot} and Proposition \ref{pro:algLCP}, $\mg$ is has a semidirect product structure
\[
\mg=\mh\ltimes \R u,
\]where $u\in\mz(\mg')$ and $\mh$ is a non-unimodular solvable Lie subalgebra of $\mg$. 
Moreover, by \eqref{eq:trace}, the trace form of $\mh$ is given by $H^{\mh}=-\theta|_{\mh}$ and, since $\dim\mmu=1$,  by Theorem \ref{teo:UUbot}(2)  we have for any $x\in\mh$ 
\begin{equation}\label{eq:xuu} [x,u]=-H^{\mh}(x) u.\end{equation}

Conversely, any Lie algebra obtained as such a semidirect product admit an LCP structure 1-dimensional flat subspace, as shown in Proposition \ref{pro:cnstr}.

Therefore, we need to determine all isomorphism classes of semidirect products of non-unimodular 4-dimensional Lie algebras $\mh$ with $\R$, where the action on $\R$ is given by the opposite of the trace form of $\mh$ as in \eqref{eq:xuu}.

To that purpose, we make use of the list of non-unimodular 4-dimensional Lie algebras $\mh$ given in \cite[Theorem 1.5]{ABDO}. In each case, we give the non-zero Lie bracket relations of $\mh$ on a basis $\{x_i\}_{i=1}^4$ and  we compute $H^\mh$ in terms of its dual basis $\{x^i\}_{i=1}^4$. Then we consider  the semidirect product $\mg=\mh\ltimes \R u$ where the representation is given by $-H^\mh$, and give explicitly the new non-zero Lie brackets  of the form $[x_i,u]$, which together with the brackets in $\mh$, determine the Lie algebra structure of $\mg$. This allows us to give the explicit isomorphism with the Lie algebras in the statement.

\begin{itemize}
\item $\mh=\mathfrak{rr}_3:$ $[x_1,x_2]=x_2$, $[x_1,x_3]=x_2+x_3$. One readily computes $H^\mh (x)=2x^1$. Hence, the new non-zero Lie bracket in $\mg$ is $[x_1,u]=-2u$. An isomorphism between $\mg$ and $\mg_{4.2}^{-2}\oplus\R x_4$ is obtained by taking $e_1=x_1$, $e_2=x_2$, $e_3=x_3$, $e_4=u$.

\item $\mh=\mathfrak{rr}_{3,\lambda}$, $\lambda\in(-1,1]$: $[x_1,x_2]=x_2$, $[x_1,x_3]=\lambda x_3$. Then $H^\mh (x)=(1+\lambda) x^1$, thus we need to add the Lie bracket
$[x_1,u]=-(1+\lambda)u$. 

If $\lambda=0$, then $\mg$ is isomorphic to $\mathfrak{e}(1,1)\oplus \R x_3\oplus \R x_4$, and if $\lambda\neq 0$ then $\mg$ is isomorphic to  $\mg_{4.5}^{p,-p-1}\oplus\R x_4$. Indeed, this is clear when  $-\frac12\leq\lambda<0$, by taking $e_i=x_i$, for $i=1,2,3$, $e_4=u$ and $p=\lambda$. For $\lambda\in(-1,-\frac12)$, one can take $p=-(1+\lambda)$ and set $e_1=x_1$, $e_2=x_2$, $x_3=u$ and $e_4=x_3$. Similarly, for $\lambda\in(0,\frac12)$ we set $p=-\lambda$ whilst for $\lambda\in[\frac12,1]$ we take $p=\lambda-1$.

\item $\mh=\mathfrak{rr'}_{3,\gamma}$, $\gamma> 0$: $[x_1,x_2]=\gamma x_2-x_3$, $[x_1,x_3]=x_2+ \gamma  x_3$. We have $H^\mh=2\gamma x^1$ and thus the new non-zero Lie bracket in $\mg$ is $[x_1,u]=-2\gamma u$. Taking $e_1=x_1$, $e_2=u$, $e_3=x_2$, $e_4=x_3$ and $p=\lambda$ shows that this Lie algebra is isomorphic to $\mg_{4.6}^{-2p,p}\oplus\R x_4$. 

\item  $\mh=\mathfrak{r}_{2}\mathfrak{r}_{2}$: $[x_1,x_2]=x_2$, $[x_3,x_4]=x_4$. This time $H^\mh= x^1+x^3$ and thus in the given basis of $\mg$ we have two new non-zero Lie brackets: $ [x_1,u]=-u=[x_3,u]$. Taking $e_1=x_i$, $i=1,2,4$,  $e_3=u$ and $e_5=-x_3$ shows that this Lie algebra is isomorphic to $\mg_{5.33}^{-1,-1}$.

\item  $\mh=\mathfrak{r}_{2}'$: $[x_1,x_3]=x_3$, $[x_1,x_4]=x_4$, $[x_2,x_3]=x_4$, $[x_2,x_4]=-x_3$. The trace form is $H^\mh= 2 x^1$ and thus the non-zero Lie bracket to be added to $\mg$ is $[x_1,u]=-2u$. Taking $e_1=x_2$, $e_2=x_3$, $e_3=x_4$, $e_4=u$ and $e_5=x_1$ shows that this Lie algebra is isomorphic to $\mg_{5.35}^{-2,0}$.

\item  $\mh=\mathfrak{r}_{4}$: $[x_4,x_1]=x_1$, $[x_4,x_2]=x_1+x_2$, $[x_4,x_3]=x_2+x_3$. We have $H^\mh= 3 x^4$ and so the only non-zero Lie bracket in $\mg$ added to those in $\mh$ is $[x_4,u]=-3u$. Taking $e_i=x_i$ for $i\le 3$, $e_4=u$ and $e_5=x_4$ gives an isomorphism between $\mg$ and $\mg_{5.11}^{-3}$. 

\item  $\mh=\mathfrak{r}_{4,\mu}$, $\mu\neq -\frac12$: $[x_4,x_1]=x_1$, $[x_4,x_2]=\mu x_2$, $[x_4,x_3]=x_2+\mu x_3$. We have $H^\mh= (1+2\mu) x^4$ and so the new non-zero Lie bracket is $[x_4,u]=-(1+2\mu)u$. If $\mu=0$, then $\mg$ is isomorphic to $\mg_{5.8}^{-1}$. Otherwise, $\mg$ is isomorphic to $\mg_{5.9}^{p,-2-p}$ for some $p$. Indeed, when $\mu\in (-\infty,-1]\cup(0,+\infty)$, the isomorphism is clear by taking $p=\frac1{\mu}$, $e_1=x_1$, $e_2=\frac{1}{\mu}x_2$, $e_3=x_3$, $e_4=u$, and  $e_5=\frac1{\mu}x_4$. Moreover, for $\mu\in(-1,0)$, $p=-2-\frac1{\mu}\geq -1$ and the isomorphism is given by choosing $e_1=u$, $e_2=\frac{1}{\mu}x_2$, $e_3=x_3$, $e_4=x_1$, and $e_5=\frac1{\mu}x_4$.

\item  $\mh=\mathfrak{r}_{4,\alpha, \beta}$, $\alpha+\beta\neq -1$; with either $-1<\alpha\leq \beta\leq 1$, $\alpha\beta\neq 0$, or  $\alpha=-1\leq \beta<0$: $[x_4,x_1]=x_1$, $[x_4,x_2]=\alpha x_2$, $[x_4,x_3]=\beta  x_3$. Then $H^\mh= (1+\alpha+\beta) x^4$, so we add the bracket $[x_4,u]=-(1+\alpha+\beta)u$. In this case $\mg$ is isomorphic to $\mg_{5.7}^{p,q,r}$. Indeed, when $\beta\leq -(1+\alpha+\beta)$  this is clear by setting $e_5=-x_4$, $e_1=-x_2$, $e_2=-x_3$, $e_3=-u$, $e_4=x_1$, $p=\alpha$, $q=\beta$, $r=-(1+\alpha+\beta)$. If this is not the case, permuting the elements $x_2,x_3,u$ in the basis of $\mg$, gives the required isomorphism.

\item $\mh=\mathfrak{r}_{4,\gamma, \delta}'$, $\delta>0$, $\gamma\neq -\frac12$: $[x_4,x_1]=x_1$, $[x_4,x_2]=\gamma x_2-\delta x_3$, $[x_4,x_3]=\delta x_2+\gamma  x_3$.  The trace form is $H^\mh= (1+2\gamma) x^4$ and thus the added Lie bracket is $[x_4,u]=-(1+2\gamma)u$. It is straightforward that $\mg$ is isomorphic to $\mg_{5.13}^{-1-2q,q,r}$. Indeed, for $\gamma\in[-1,0]$ the isomorphism is obvious for  $q=\gamma$, $r=\delta$; otherwise one should choose $q=\frac{\gamma}{-1-2\gamma}$, $r=\delta$.

\item  $\mh=\mathfrak{d}_{4,\lambda}$, $\lambda\geq \frac12$: $[x_1,x_2]=x_3$, $[x_4,x_3]=x_3$,  $[x_4,x_1]=\lambda x_1$, $[x_4,x_2]=(1-\lambda) x_2$.  We have $H^\mh= 2 x^4$ and only non-zero Lie bracket in $\mg$ to be added is 
$[x_4,u]=-2u$. It follows that $\mg$ is isomorphic to $\mg_{5.19}^{p,-2p-2}$ by taking $e_5=\frac1{\lambda}x_4$, $e_1=x_i$, $i=1,2,3$, $e_4=u$, and $p=\frac{1-\lambda}{\lambda}$.

\item $\mh=\mathfrak{d}'_{4,\delta}$, $\delta>0$: $[x_1,x_2]=x_3$, $[x_4,x_1]=\frac{\delta}2x_1-x_2$,  $[x_4,x_3]=\delta x_3$, $[x_4,x_2]=x_1+\frac{\delta}2 x_2$. The trace form is $H^\mh= 2\delta x^4$ and $
[x_4,u]=-2\delta u$ is the non-zero Lie bracket to be added to $\mg$. It follows that $\mg$ is isomorphic to $\mg_{5.25}^{p,4p}$ by taking  $p=\delta/2$, $e_5=x_4$, $e_1=x_1$, $e_2=-x_2$, $e_3=x_3$, and $e_4=u$.

\item $\mh=\mh_4$: $[x_1,x_2]=x_3$, $[x_4,x_3]=x_3$,  $[x_4,x_1]=\frac12 x_1$, $[x_4,x_2]=x_1+\frac12 x_2$.  We have $H^\mh= 2x^4$ and the only non-zero Lie bracket in $\mg$ added to those in $\mh$ is 
$[x_4,u]=-2u$. It follows that $\mg$ is isomorphic to $\mg_{5.23}^{-4}$ by taking  $e_5=2x_4$,  $e_i=2x_i$, $i=1,3$, $e_2=x_2$, and $e_4=u$.
\end{itemize}
\end{proof}

The summary of the $5$-dimensional case is given in Table \ref{table:dim5-1} in the Appendix.

\begin{remark}
The restrictions $r>0$ for $\mg_{5.13}^{-1-2q,q,r}$, $q>0$ for $\mg_{5.16}^{-1,q}$,  $p\geq 0$, $r>0$ for $\mg_{5.17}^{p,-p,r}$, and $p>0$ for $\mg_{5.25}^{p,4p}$ that appear in Table \ref{table:dim5-1} are not included in \cite[Appendix A]{Bo16}. However, it is easy to check that, in all cases, the Lie algebra with a negative parameter is isomorphic to the one with opposite (positive) parameter just by switching two elements of the basis.\end{remark}

\medskip

\section{Lattices}\label{ss:lat}
In this section we study the existence of lattices on the simply connected solvable unimodular Lie groups corresponding to the LCP Lie algebras appearing in the previous section. The final objective is to  study the associated LCP solvmanifolds.

\subsection{Lattices on almost abelian Lie groups}
We start this section by recalling the following result by Bock:
\begin{teo}\cite{Bo16}\label{teo:Bock} A unimodular almost abelian Lie group $G=\R\ltimes_\rho \R^{n-1}$ with Lie algebra $\mg=\R b\ltimes\mk$ admits a lattice if and only if there is a basis $\mathcal{B}$ of $\mk=\R^{n-1}$ and $t_0\neq 0$ such that the matrix of $\rho(t_0)=e^{t_0 \ad_{b}}\in \Aut(\mk)$ in that basis is in $\SL(n-1,\Z)$. 
\end{teo}
In this case a lattice can be given by $\Gamma=t_0\Z\ltimes_\rho \Gamma_0$, where $\Gamma_0\simeq \Z^{n-1}$ is the lattice of $\R^{n-1}$ spanned by $\mathcal{B}$.

Recall from Example \ref{ex:amal1} that the amalgamated product of two almost abelian Lie algebras is again almost abelian. The next example studies the existence of lattices on such products.

\begin{ex}\label{ex:amal2}

Let $\mg$  be  the amalgamated product of $(\mg_1,g_1,\theta_1,\mmu_1)$ and $(\mg_2,g_2,\theta_2,\mmu_2)$, where both $\mg_1$ and $\mg_2$ are unimodular almost abelian Lie algebras, with $\mg_i=\R b_i\ltimes_{C_i} \R^{n_i}$ for $i=1,2$. We know that $\mg$ is again a unimodular almost abelian Lie algebra, as it can be written as $\mg=\R b\ltimes_C (\R^{n_1}\oplus\R^{n_2})$, with $b$ and $C$ as in \eqref{eq:b} and \eqref{eq:C}, respectively. Let us assume now that $G_1$ and $G_2$, the simply connected Lie groups associated to $\mg_1$ and $\mg_2$ respectively, have lattices. According to Theorem \ref{teo:Bock}, there exist $t_1,t_2\in\R$ such that $e^{t_1C_1}\in \SL(n_1,\Z)$ and $e^{t_2C_2}\in \SL(n_2,\Z)$, for certain bases $\mathcal B_1$ and $\mathcal B_2$ of $\R^{n_1}$ and $\R^{n_2}$, respectively. We may assume that $t_1, t_2>0$. We will give some conditions to ensure that $G$, the simply connected Lie group with Lie algebra $\mg$, admits lattices. 

Let us assume that the quotient $\frac{t_2|\theta_2|}{t_1|\theta_1|}$ is a rational number. Then there exists $t>0$ such that
\begin{equation}\label{eq:latt}  t\frac{|\theta_2|}{(|\theta_1|^2+|\theta_2|^2)^\frac12} = k_1t_1, \quad   t\frac{|\theta_1|}{(|\theta_1|^2+|\theta_2|^2)^\frac12} = k_2t_2, \quad \text{for some } k_1,k_2\in \N. \end{equation}
 By \eqref{eq:C}, the matrix of $e^{tC}$, expressed in the basis $\mathcal{B}_1\cup \mathcal{B}_2$ of $\R^{n_1}\oplus \R^{n_2}$, is given by
\[
e^{tC}= \left[ \begin{array}{c|c} e^{k_1t_1C_1} & \\ \hline
&  e^{k_2t_2C_2}
\end{array} \right]=\left[ \begin{array}{c|c} (e^{t_1C_1})^{k_1} & \\ \hline
&  (e^{t_2C_2})^{k_2}\end{array} \right]\in \SL(n_1+n_2,\Z).
\]
Therefore, $G$ admits lattices, due to Theorem \ref{teo:Bock}.

In particular, if the simply connected Lie group corresponding to a unimodular almost abelian LCP Lie algebra admits lattices, then the simply connected Lie group whose Lie algebra is the amalgamated product of this LCP Lie algebra with itself, also admits lattices. In this way, one obtains (compact) LCP solvmanifolds with arbitrarily large flat space.
\end{ex}

From Proposition \ref{pro:ubotgeq2}, we know that on any $n$-dimensional unimodular solvable LCP Lie algebra, the flat factor has dimension $\le n-2$, and by Proposition \ref{pro:ubot2}, if the equality holds, the algebra is almost abelian. This allows us to use Theorem \ref{teo:Bock} in order to study LCP solvmanifolds with flat factor of codimension 2.

\begin{pro}\label{pro:dim=n-2}
Let  $\mg$ be a unimodular solvable Lie algebra of dimension $n$, with corresponding simply connected Lie group $G$. Assume that $G$ admits a lattice $\Gamma$ and there is a non-degenerate LCP structure $(g, \theta,\mmu)$ on $\mg$ satisfying $\dim \mmu=n-2$, so that $\Gamma\bs G$ is a compact LCP solvmanifold. Then either $\dim \mg=3$ or $\dim\mg=4$.
\end{pro}
\begin{proof}
Assume that $(\mg, g, \theta,\mmu)$ is a non-degenerate LCP Lie algebra satisfying $\dim \mmu=n-2$. By Propositions \ref{pro:ubotgeq2} and \ref{pro:ubot2}, we know that 
$\mmu^\bot$ has an orthonormal basis $\{b,x\}$ such that $\theta=\mu g(b,\cdot)$ for some $\mu>0$, $\mk:=\ker \theta=\R x\oplus \mmu$ is an abelian ideal and
\begin{equation*}
\ad_b=-(n-2)\mu x\otimes x+\mu\Id_\mmu+B,
\end{equation*}where $B$ is a fixed skew-symmetric endomorphism of $\mmu$.

In particular, $\mg=\R b\oplus \mk$ is an orthogonal direct sum and $\mg$ is almost abelian. Thus $G\simeq \R\ltimes_\rho \R^{n-1}$ where, if we write $\R^{n-1}=\R x\oplus \mmu$, we have
\begin{equation}
\label{eq:rhot}
\rho(t)=e^{t \ad_b}= \left( \begin{array}{c|ccc}
e^{-t\mu(n-2)}&&0\\
\hline\\
&&e^{t\mu}\cdot  e^{tB}\\
\\
\end{array}\right),\qquad t\in\R.
\end{equation} 

Assume that  $G$ admits a lattice. By Theorem \ref{teo:Bock}, there is some $t_0\neq 0$ such that $\rho(t_0)$ is conjugated to an element in $\SL(n-1,\Z)$, that is, $Q\rho(t_0)Q^{-1}=:Z\in\SL(n-1,\Z)$ for some $Q\in \GL(n-1,\R)$. 

It is clear from  \eqref{eq:rhot} that $\R x$ and  $\mmu$ are preserved by $\rho(t_0)$. Hence, $\R Qx$ and $Q\mmu$ are invariant under $Z$, and $Z|_{Q\mmu}$ is orthogonal with respect to the (restriction of the) inner product $e^{-2t\mu}g(Q^{-1}\cdot,Q^{-1}\cdot)$. Then, by \cite[Proposition 3]{MMP}, either $n=3$ or $n=4$.
\end{proof}

An immediate consequence of this result is the following:
\begin{cor}\label{cor:dim=n-2}Let $G$ be a simply connected solvable Lie group of dimension $n\geq 5$ with Lie algebra $\mg$. If $\mg$ admits an LCP structure $(g,\theta,\mmu)$ with $\dim \mmu=n-2$, then $G$ has no lattices.
\end{cor}

\subsection{Low-dimensional case}

In this section we will determine which of the simply connected solvable Lie groups corresponding to the low-dimensional LCP Lie algebras that appeared in \S \ref{sec:classification} 
have lattices. In order to do so, we will use mainly results in \cite{Bo16}, but in some cases we will prove the (non-)existence directly. For this we will use the following result about polynomials with integer coefficients.

\begin{lm}\label{lm:irred}
Let $P\in\Z[X]$ be a polynomial with integer coefficients.
  \begin{enumerate}
      \item If $P$ is irreducible, then all its roots in $\C$ are simple. 
      \item If $|P(0)|=1$, $P$ is monic and has a double root different from $\pm1$, then it has at least two double roots.
  \end{enumerate}
 \end{lm}
\begin{proof} (1) It is well known that a polynomial irreducible in $\Z[X]$ is also irreducible in $\Q[X]$. Assume that $P$ has a double root $a$ and let $\mu_a$ be the minimal polynomial of $a$ in $\Q[X]$. Since $P'(a)=0$, $\deg(\mu_a)\le \deg(P)-1$. As $\Q[X]$ is principal, $P$ is a non-constant multiple of $\mu_a$, so it is reducible, contradicting the hypothesis.

(2) Let $a$ be a double root of $P$. From the above, $P$ is divisible by $\mu_a^2$, so it is enough to show that $\deg(\mu_a)\ge 2$. Indeed, if $\deg(\mu_a)=1$, $a$ would be rational, thus integer as $P$ is monic, and finally $a=\pm1$ because $a$ divides $|P(0)|$.
\end{proof}

\begin{cor}\label{cor:double}
    Let $A\in \sl(q)$ be a trace-free matrix and let $G$ be the simply connected Lie group with algebra $\mg=\R b\ltimes_A\R^q$, where the action is defined by $\ad_b|_{\R^q}=A$. If all eigenvalues of $A$ are real and the characteristic polynomial of $A$ has exactly one multiple root $\lambda\ne 0$, then $G$ has no lattices.
\end{cor}

\begin{proof}
    If $G$ has a lattice, Theorem \ref{teo:Bock} shows that there exists a non-zero real number $t_0$ such that $e^{t_0A}$ is conjugate to a matrix in $\mathrm{SL}(q,\Z)$. Then the characteristic polynomial, $P$, of $e^{t_0A}$ is monic, has integer coefficients, satisfies $P(0)=\det(e^{t_0A})=e^{t_0\tr(A)}=1$ and has a double root $e^{t_0\lambda}$ different from $\pm 1$. The hypothesis on the characteristic polynomial of $A$ implies that all other roots of $P$ are simple, which contradicts Lemma \ref{lm:irred}(2). Thus $G$ has no lattices.
\end{proof}

We will now analyse the existence of lattices in all simply connected Lie groups whose corresponding LCP Lie algebras were classified in Section \ref{sec:classification}. The analysis will be made according to the dimension of the Lie algebra. The information obtained is included in the last column of the tables in the Appendix.

\textbf{Case 1.} $\dim \mg=3$.
By Proposition \ref{pro:dim3}, the only unimodular solvable LCP Lie algebra in dimension 3 is $\mathfrak{e}(1,1)$. Let us denote by $E(1,1)$ the simply connected Lie group associated to $\mathfrak{e}(1,1)$. The fact that $E(1,1)$ admits lattices is well-known. For the sake of completeness we exhibit lattices in $E(1,1)$, using Theorem \ref{teo:Bock}.

The Lie algebra $\mathfrak{e}(1,1)$ can be written as $\mathfrak{e}(1,1)=\R\ltimes \R^2$, where the action of $\R$ on $\R^2$ is given by the matrix \[A:=\begin{bmatrix}
1 & 0 \\ 0 & -1\end{bmatrix}.\] For $m\in \N$, $m\geq 3$, set $t_m:=\ln \left(\frac{m+\sqrt{m^2-4}}{2}\right)$. Then, since $e^{t_m}+e^{-t_m}=m$, it is easy to check that $\exp(t_m A)$ is conjugate to 
\[E_m:=\begin{bmatrix} 0 & -1 \\ 1 & m \end{bmatrix}\in \operatorname{SL}(2,\Z),\] so there exists $Q\in \operatorname{GL}(2,\R)$ with $E_m=Q\exp(t_m A)Q^{-1}$. Then
$\Gamma_m:=t_m\Z \ltimes Q^{-1}\Z^{2}$ is a lattice in $E(1,1)$.

The lattices $\Gamma_m$ above are pairwise non-isomorphic, since it is readily verified that 
\[ \Gamma_m/[\Gamma_m,\Gamma_m]\cong \Z\oplus \Z_{m-2}. \]
Thus, the solvmanifolds $\Gamma_m\backslash E(1,1)$ are pairwise non-homeomorphic, since $\pi_1(\Gamma_m\backslash E(1,1))$ is isomorphic to $\Gamma_m$.

\textbf{Case 2.} $\dim \mg=4$. By Proposition \ref{pro:dim4} we have the following 4 possibilities:
\begin{itemize}
    \item $\mathfrak{e}(1,1)\oplus\R$: clearly $E(1,1)\times \R$ admits lattices, since $E(1,1)$ does. 

\item $\mg_{4.2}^{-2}$: the corresponding simply connected Lie group $G_{4.2}^{-2}$ does not admit lattices, due to \cite[Theorem 7.1.1]{Bo16}. Notice that the Lie algebra $\mg_{4.2}^{-2}$ is sometimes referred to as $\mg_{4.2}$ in \cite{Bo16}.
\smallskip

\item $\mg_{4.5}^{p,-p-1}$: the corresponding simply connected Lie groups $G_{4.5}^{p,-p-1}$ admit lattices for some values of the parameter $p\in[-\frac12,0)$, according to \cite[Theorem 6.2 and Table A.1]{Bo16}. However, we can rule out the value $p=-\frac12$. Indeed, for this value 
the eigenvalues of $\ad_b|_\mk$ are $1$, $-\frac12$ and $-\frac12$, so by Corollary \ref{cor:double}, the simply connected Lie group $G_{4.5}^{-\frac12,-\frac12}$ has no lattices.

\item $\mg_{4.6}^{-2p,p}$: it follows from \cite[\S 3.2.2]{AO} that for some values of $p$ the corresponding simply connected Lie group $G_{4.6}^{-2p,p}$ has lattices.

\end{itemize}

\textbf{Case 3.} $\dim \mg=5$. We start by studying the decomposable Lie algebras appearing in Propositions \ref{pro:dim5-3}, \ref{pro:dim5-2} and \ref{pro:dim5-1}:

\begin{itemize}
    \item  $\mathfrak{e}(1,1)\oplus\R^2$: clearly $E(1,1)\times \R^2$ admits lattices, since $E(1,1)$ does. 
    \item $\mg_{4.2}^{-2}\oplus \R$: the corresponding simply connected Lie group admits no lattices due to \cite[Theorem 7.1.1]{Bo16}. 
    \item $\mg_{4.5}^{p,-p-1}\oplus \R$: $G_{4.5}^{p,-p-1}\times \R$ admit lattices for some values $p$, because $G_{4.5}^{p,-p-1}$ do, as explained above. 
    Note however that for $p=-\frac12$,
    the eigenvalues of $\ad_b|_\mk$ are $0$, $1$, $-\frac12$ and $-\frac12$, so by Corollary \ref{cor:double}, the simply connected Lie group $G_{4.5}^{-\frac12,-\frac12}\times \R$ has no lattices.
    \item $\mg_{4.6}^{-2p,p}\oplus \R$: $G_{4.6}^{-2p,p}\times \R$ admit lattices for some values of $p$, since $G_{4.6}^{-2p,p}$ have lattices \cite[\S 3.2.2]{AO}.
\end{itemize}

Finally, we treat the indecomposable Lie algebras listed in Table \ref{table:dim5-1}:
\begin{itemize}
    \item $\mg_{5.7}^{p,q,r}$: The  corresponding simply connected Lie groups $G_{5.7}^{p,q,r}$ admit lattices for some values of the parameters (see \cite[Theorem 7.2.1]{Bo16}). However, we can be more precise for some particular values of the parameters.
    
    In fact, these Lie algebras are almost abelian, and the action on the abelian ideal is given by a diagonal matrix $A$ with eigenvalues $p,q,r,-1$, where $pqr\neq 0$, $p+q+r=1$ and $-1\leq p\leq q\leq r\leq 1$. Hence if $p=q$ or $q=r<1$, then all roots of the characteristic polynomial of $A$ are real, and only one of them is multiple. Therefore, in these cases, the corresponding simply connected Lie groups $G_{5.7}^{p,q,r}$ have no lattices due to Corollary \ref{cor:double}. Notice that for $r=q=1$ we have $p=-1$ and  $G_{5.7}^{-1,1,1}$ has lattices by \cite[Theorem 7.2.1 (iii)]{Bo16}.
    
    \item $\mg_{5.8}^{-1}$: the Lie group $G_{5.8}^{-1}$ admits lattices, due to \cite[Theorem 7.2.2]{Bo16}.
    \item $\mg_{5.9}^{p,-2-p}$: the Lie groups $G_{5.9}^{p,-2-p}$ do not admit lattices (see \cite[Theorem 7.2.3]{Bo16}).
    \item $\mg_{5.11}^{-3}$: the Lie group $G_{5.11}^{-3}$ does not admit lattices (see \cite[Theorem 7.2.4]{Bo16}).
    \item $\mg_{5.13}^{-1-2q,q,r}$: for any $r> 0$, Proposition \ref{pro:dim5-3} shows that the Lie groups $G_{5.13}^{-\frac13,-\frac13,r}$ admit LCP structures with flat factor of dimension 3, so by Corollary \ref{cor:dim=n-2}, they have no lattices.
    However, for $q\neq -\frac13$, there are some values of the parameters $q,r$ for which $G_{5.13}^{-1-2q,q,r}$ admit lattices by \cite[Proposition 7.2.5]{Bo16}.
    \item $\mg_{5.16}^{-1,q}$: the Lie groups $G_{5.16}^{-1,q}$  do not admit lattices (see \cite[Theorem 7.2.10]{Bo16}).
    \item $\mg_{5.17}^{p,-p,r}$: $G_{5.17}^{p,-p,r}$ admit lattices for some values of $p\neq 0$, as shown in \cite[Theorem 7.2.12]{Bo16}.
    \item $\mg_{5.19}^{p,-2p-2}, \mg_{5.23}^{-4}$ and $\mg_{5.25}^{p,4p}$:
    the associated Lie groups $G_{5.19}^{p,-2p-2}$, $G_{5.23}^{-4}$ and $G_{5.25}^{p,4p}$ do not have lattices, according to \cite[Theorem 7.2.16]{Bo16}.
    \item $\mg_{5.33}^{-1,-1}$ and $\mg_{5.35}^{-2,0}$: it follows from \cite[Propositions 7.2.20-21]{Bo16} that the associated simply connected Lie groups $G_{5.33}^{-1,-1}$ and $G_{5.35}^{-2,0}$ admit lattices.
\end{itemize}
\medskip

\appendix
\section{Tables of low-dimensional LCP Lie algebras}

This appendix contains the structure constants of the unimodular solvable LCP Lie algebras up to dimension 5. 
In each table, the 3rd column gives the possible dimensions of the flat factor of an LCP structure, as determined in Section \ref{sec:classification}. The 4th column says whether the corresponding simply connected Lie group has lattices or not. Notice that when the Lie algebra belongs to a family, it is in general not possible to determine explicitly the set of parameters for which the corresponding simply connected Lie groups admit lattices (except when one can show that none of them does).
\medskip

\begin{table}[h] 
\begin{tabular}{|c|c|c|c|} \hline
Lie algebra & Brackets & $\dim \mmu$  &Lattices \\ \hline 
$\mathfrak{e}(1,1)$ & $ [e_1,e_2]=e_2,\, [e_1,e_3]=-e_3$ & 1 &yes  \\ \hline
\end{tabular} 
\vskip 2mm
\caption{$3$-dimensional LCP Lie algebra}
\label{table:dim3}
\end{table}
\vspace{-0.3cm}

\begin{table}[htb] 
\begin{tabular}{|c|c|c|c|} \hline
Lie algebra & Brackets & $\dim\mmu$ & Lattices  \\ 
\hline
$\mathfrak{e}(1,1)\oplus \R$ & $ [e_1,e_2]=e_2,\, [e_1,e_3]=-e_3$ & 1& yes \\
\hline 
$\mg_{4.2}^{-2}$
& $[e_1,e_2]=e_2, \, [e_1,e_3]=e_2+ e_3,$ & 1& no \\ 
&$[e_1,e_4]=-2e_4$&& \\
\hline 
$\mg_{4.5}^{p,-p-1}$
& $[e_1,e_2]=e_2, \, [e_1,e_3]=pe_3, $  & 1 or &for some $p\neq -\frac12$\\ 
&$ [e_1,e_4]=-(p+1)e_4$, $-\frac12\leq p<0$&2 for $p=-\frac12$  &\\
\hline
$\mg_{4.6}^{-2p,p}$ 
& $[e_1,e_2]=-2p e_2, \, [e_1, e_3]=pe_3-e_4,$ & 1 or 2 &for some $p$ \\ 
&  $[e_1,e_4]=e_3+pe_4,$ $p>0$&&\\ \hline
\end{tabular} 
\vskip 2mm
\caption{$4$-dimensional LCP Lie algebras}
\label{table:dim4}
\end{table}

\begin{table}[htb] 
\begin{tabular}{|c|c|c|c|} \hline
Lie& Brackets & $\dim\mmu$ & Lattices\\ 
 algebra &  &  & \\ 
 \hline
$\mathfrak{e}(1,1)\oplus \R^2$ & $ [e_1,e_2]=e_2,\, [e_1,e_3]=-e_3$ & 1& yes \\
\hline 
$\mg_{4.2}^{-2}\oplus \R$
& $[e_1,e_2]=e_2, \, [e_1,e_3]=e_2+ e_3,$ & 1& no \\ 
&$[e_1,e_4]=-2e_4$&& \\
\hline 
$\mg_{4.5}^{p,-p-1}\oplus \R$
& $[e_1,e_2]=e_2, \, [e_1,e_3]=pe_3, $  & 1 or &for some \\ 
&$ [e_1,e_4]=-(p+1)e_4$, $-\frac12\leq p<0$&2 for $p=-\frac12$  &$p\neq -\frac12$\\
\hline
$\mg_{4.6}^{-2p,p}\oplus\R$ 
& $[e_1,e_2]=-2p e_2, \, [e_1, e_3]=pe_3-e_4,$ & 1 or 2 &for some $p$ \\ 
&  $[e_1,e_4]=e_3+pe_4,$ $p>0$&&\\ 
\hline
$\mg_{5.7}^{p,q,r}$  
& $[e_5,e_1]=pe_1, [e_5,e_2]=qe_2, [e_5,e_3]=re_3, $ &1 or& for some  \\ 
& $[e_5,e_4]=-e_4$ &2 if  $p\neq r$ and &  $p\neq q\neq r$\\
& $pqr\neq 0, \, p+q+r=1$&$q\in\{p,r\}$ or&or\\
&$-1\leq p\leq q\leq r\leq 1$&3 if $p=q=r$& $q=r=1$\\
\hline 
$\mg_{5.8}^{-1}$ 
& $[e_5,e_1]=e_1, \, [e_5,e_3]=e_2, \, [e_5,e_4]=-e_4$ &1 &yes  \\ 
\hline 
$\mg_{5.9}^{p,-2-p}$ 
& $[e_5,e_1]=pe_1, \, [e_5,e_2]=e_2,\,  [e_5,e_3]=e_2+e_3, $ &1 or &no  \\ 
& $[e_5,e_4]=(-2-p)e_4$, \; $p\geq -1$ &2 if $p=-1$ &\\
\hline 
$\mg_{5.11}^{-3}$ 
& $[e_5,e_1]=e_1, \, [e_5,e_2]=e_1+e_2,, $ &1& no  \\ 
& $[e_5,e_3]=e_2+e_3,\,[e_5,e_4]=-3e_4 $ && \\
\hline 
$\mg_{5.13}^{-1-2q,q,r}$ 
& $[e_5,e_1]=e_1, \, [e_5,e_2]=qe_2-re_3, $ &1 or& for some   \\ 
& $ [e_5,e_3]=re_2+qe_3,\,[e_5,e_4]=(-1-2q)e_4 $&2 or& $q\neq -\frac13,r$ \\
& $r>0$, $q\in[-1,0]$, $q\neq -\frac12$&3 if $q=-\frac13$&\\
\hline 
$\mg_{5.16}^{-1,q}$ 
& $[e_5,e_1]=e_1, \, [e_5,e_2]=e_1+e_2, $ &2& no   \\ 
&$[e_5,e_3]=-e_3-qe_4,\,[e_5,e_4]=qe_3-e_4$,\;$q>0$& &\\
\hline 
$\mg_{5.17}^{p,-p,r}$ 
& $[e_5,e_1]=pe_1-e_2, \, [e_5,e_2]=e_1+pe_2, $ &2 if $p\neq 0$& for some  \\ 
& $ [e_5,e_3]=-pe_3-re_4,\,[e_5,e_4]=re_3-pe_4$&&  $p,r$\\
& $p\geq 0$, $r>0$&&\\
\hline 
$\mg_{5.19}^{p,-2p-2}$
& $[e_1,e_2]=e_3, \, [e_5,e_1]=e_1, [e_5,e_2]=p e_2,$ &1 &no \\ 
& $[e_5,e_3]=(p+1)e_3, \, [e_5,e_4]=-2(p+1)e_4$,& &\\ 
&$p\neq -1$&&\\
\hline 
$\mg_{5.23}^{-4}$ 
& $[e_1, e_2]=e_3,\, [e_5,e_1]=e_1, \, [e_5,e_2]=e_1+e_2, $ &1& no  \\ 
& $[e_5,e_3]=2e_3, \,[e_5,e_4]=-4e_4$ &  &\\
\hline 
$\mg_{5.25}^{p,4p}$  
& $[e_1, e_2]=e_3,\, [e_5,e_1]=pe_1+e_2,  $ &1 & no  \\ 
& $[e_5,e_2]=-e_1+pe_2, \,[e_5,e_3]=2pe_3,$ & &  \\
&$[e_5,e_4]=-4pe_4$, $p>0$&& \\
\hline 
$\mg_{5.33}^{-1,-1}$
& $[e_1,e_2]=e_2,\, [e_1,e_3]=-e_3,$ &1 &yes \\ 
& $ [e_5,e_3]=e_3,\, [e_5,e_4]=-e_4 $& &\\
\hline
$\mg_{5.35}^{-2,0}$ & $[e_1,e_2]=e_3,\, [e_1,e_3]=-e_2,\, [e_5,e_2]=e_2, $ &1 or 2 & yes  \\
& $ [e_5,e_3]=e_3,\, [e_5,e_4]=-2e_4$ &  &\\
\hline
\end{tabular} 
\vskip 2mm
\caption{$5$-dimensional LCP Lie algebras}
\label{table:dim5-1}
\end{table}

\FloatBarrier

\bibliographystyle{plain}
\bibliography{biblio}

\end{document}